\documentclass[a4paper, 10pt]{amsart}
\usepackage{latexsym}
\usepackage{amscd}
\usepackage{fancyhdr}
\usepackage{amsmath, amssymb}
\usepackage[ansinew]{inputenc}
\usepackage{verbatim}
\usepackage{stackrel}
\usepackage{amsmath}
\usepackage{amsfonts}
\usepackage{amssymb}
 \swapnumbers \theoremstyle{plain}
\newtheorem*{lemma*}{Lemma}
\newtheorem{lemma}[subsection]{Lemma}
\newtheorem*{theorem*}{Theorem}
\newtheorem{theorem}[subsection]{Theorem}
\newtheorem*{proposition*}{Proposition}
\newtheorem{proposition}[subsection]{Proposition}
\newtheorem*{corollary*}{Corollary}
\newtheorem{corollary}[subsection]{Corollary}
\theoremstyle{definition}
\newtheorem*{definition*}{Definition}
\newtheorem{definition}[subsection]{Definition}
\newtheorem*{example*}{Example}

\newtheorem*{algorithm*}{Algorithm}
\newtheorem*{remark*}{Remark}
\newtheorem{remark}[subsection]{Remark}

\def\al{\alpha}
\def\be{\beta}
\def\ga{\gamma}
\def\Ga{Gamma}
\def\de{\delta}

\def\ve{\varepsilon}

\def\la{\lambda}

\def\si{\sigma}

\def\ta{\tau}

\def\vh{\varphi}

\def\om{\omega}

\def\Ga{\Gamma}

\def\Ph{\Phi}

\def\Om{\Omega}

\def\C{\mathbb{C}}

\def\R{\mathbb{R}}

\def\p{\partial}
\def\x{\times}

\renewcommand{\o}{\circ}
\def\<{\langle}
\def\>{\rangle}

\let\on=\operatorname

\title[Orbit spaces and leaf spaces of foliations as generalized manifolds]
{Orbit spaces and leaf spaces of foliations as generalized manifolds
}

\author[M.\ V.\ Losik]
{Mark\ V.\ Losik\,$^*$}\thanks{${ }^*$ The paper contains
unpublished results by Mark\,V.\,Losik (27.01.1935--12.05.2013). The
paper is made up of 3 tex-files found on the computer of M.V.\,Losik
by his former student A.\,Galaev, who edited the text, added some
remarks in the form of footnotes and added the abstract.}

\address{M.\ V.\ Losik: Saratov State University, Saratov, Russia}

\email{anton.galaev@uhk.cz (A. Galaev)}

\begin{document}

\begin{abstract}

The paper gives a categorical approach to generalized manifolds such
as orbit spaces and leaf spaces of foliations. It is suggested to
consider these spaces as sets equipped with some additional
structure which generalizes the notion of atlas. The approach is
compared with the known ones that use the Grothendieck topos, the
Haefliger classifying space and the Connes non-commutative geometry.
The main aim of this paper is to indicate that the suggested
approach permits to simplify essentially the modern theory of the
leaf space of foliations and the orbit space of diffeomorphism
groups, and to obtain some new results on the characteristic classes
of foliations. As an application, it is shown that the first Chern
class is non-trivial for the leaf space of the Reeb foliation and
its geometrical meaning is indicated.

\end{abstract}

%\thanks{}
%\keywords{}
%\subjclass[2000]{57D30, 57R32, 22A22, 55R40, 57D20}

%\date{\today}

\maketitle

\section{Introduction}
Such objects as leaf spaces of a foliations and orbit spaces of
diffeomorphism groups are subjects of intensive study during last
several decades. There is the vast literature devoted to the
investigation of these spaces. They  have a very complicated
structure, in particular, often their topology is very bad. By this
reason usually such a space is substituted by some model which
'encodes' the information about this space. For example, for the
leaf spaces of foliations there are three approaches to construct
such models \cite{MoModels}, namely, by means of the Grothendieck
topos, the Haefliger classifying space and the Connes
non-commutative geometry. Central to all these approaches is the
construction of a smooth \'etale groupoid $G$. The three approaches
above then become special instances of the general procedure of
associating to a smooth groupoid a classifying topos $Sh(G)$, a
classifying space $BG$, and a convolution algebra $C^\infty_c(G)$.
Thus the study  of the leaf space of a foliation  is included in the
more general theory of smooth groupoids. In the all of these
examples the leaf space of a foliation or the orbit space of a
diffeomorphism group are present only virtually.

In contrast to these approaches we consider the objects above as
sets equipped with some additional structure which generalizes the
notion of atlas. Moreover, we construct a category containing
among its objects both manifolds and all objects above. A typical
example of such an object is an orbifold. Thus, our construction
is based on the idea of the atlas consisting of charts in
contrast, for example, to the non-commutative geometry which is
based on the idea of  non-commutative, in general, algebra of
functions. Although the topology of basic sets of objects for the
category above is  often bad or trivial but the smooth structure
is very rich. First such an approach was proposed in the paper
\cite{L5}. The more general categorical construction which
contains all categories we need next is proposed in the paper
\cite{L6}. Note that the objects of these categories possess the
natural structure of Grothendieck's topology  and, in many cases,
they are defined uniquely by some smooth \'etale groupoid. This
allows us to compare our approach with the known ones.

The  main aim of this paper is to indicate that our approach permits
to simplify essentially the modern theory of the leaf space of
foliations and the orbit space of diffeomorphism groups, and to
obtain some new results on the characteristic classes of foliations.

The characteristic classes of  a foliation $\mathfrak F$ of
codimension $n$ on a smooth manifold $M$ are well known. Moreover,
it is known that they can be defined as cohomology classes on the
spaces of leaves of foliations. Usually, the characteristic
classes above are derived from the relative cohomologies
$H^*(W_n,O(n))$ and $H^*(W_n,\text{GL}(n,\R))$ of the Lie algebra
of formal vector fields with respect to the orthogonal group
$O(n)$ or from the Pontrjagin classes of the principal
$\text{GL}(n,\R)$-bundle associated to the space of leaves of
foliation. All these characteristic classes can be obtained as the
cohomology classes of the \v{C}ech-De Rham cohomology theory
constructed in \cite{Cr}. Although there are explicit formulas for
all these characteristic classes, there is no information about
non-triviality of the characteristic classes above. In the paper
\cite{L5} it was shown that the formal first Chern class of the
cohomology $H^*(W_1,{\rm GL}(1,\R))$ is in general non-trivial in some
specific cohomology theory. One of  the aims of this paper is to
prove that the first Chern class is non-trivial for the leaf space
of the Reeb foliation and to indicate its geometrical meaning.

Unless specified, all manifolds next are finite dimensional
Hausdorff manifolds of $C^\infty$-class. Throughout the paper all
categories are proposed to be small.

\section{$\mathcal C$-spaces}

\subsection{Definition of $\mathcal C$-spaces}\label{secdefCsp}

Let $\mathcal C$ be a category. Next we denote the space of objects
of $\mathcal C$ by $\mathcal C_0$ and the set of morphisms of
$\mathcal C$ by $\mathcal C_1$. Let  $a\in\mathcal C_0$. A set $S$
of $\mathcal C$-morphisms $f:b\to a$ is an $a$-{\bfseries sieve}, if
the composition of an arbitrary $\mathcal C$-morphism $c\to b$ and
$f$ belongs to $S$. The minimal $a$-sieve containing a given set $T$
of $\mathcal C$-morphisms with the target $a$ is called an
$a$-{\bfseries sieve generated by} $T$. The set of $\mathcal
C$-morphisms $c\to b$ whose compositions with $f$ belong to $S$, is
called a {\bfseries restriction} of $S$ to $b$ (with respect to $f$)
and denoted by $f^*S$.

Recall that a {\bfseries Grothendick topology} on a category
$\mathcal C$ is constituted by assigning to each $a\in\mathcal C_0$
the set $\on{Cov}(a)$ of $a$-sieves, called {\bfseries covers}, such
that the following axioms hold:
\begin{enumerate}
 \item For any $a\in\mathcal C_0$, the $a$-sieve generated by identity morphism $1_a$ is a cover;
\item For any $\mathcal C$-morphism $f:b\to a$ and $S\in\on{Cov}(a)$, we have $f^*S\in\on{Cov}(b)$;
\item Let $S\in\on{Cov}(a)$ and let $R$ be an $a$-sieve. Then $R\in\on{Cov}(a)$ if, for any $\mathcal C$-morphism $f:b\to a$ from $S$, we have $f^*R\in\on{Cov}(b)$.
\end{enumerate}
A category $\mathcal C$ with a Grothendieck topology on $\mathcal C$
is called a {\bfseries site}.

Referring  to the paper \cite{L6} for the general exposition we
consider the following five categories.

\begin{itemize}

\item[1] The category $\mathcal M_n$ with objects smooth manifolds of
dimension $n$ and morphisms  local diffeomorphisms (i.e. \'etale
morphisms) of one manifold into another one\footnote{Here a smooth
map $f:M\to N$ between two manifolds is called a local
diffeomorphism if each point of $M$ has an open neighborhood $U$
such that $f$ restricted to $U$ is a diffeomorphism onto~$f(U)$.}.

\item[2] The full subcategory $\mathcal D_n$ of the category $\mathcal
M_n$ with objects open submanifolds of $\mathbb R^n$.

\item[3]  The category $\mathcal M_{\infty}$ with objects smooth
manifolds with model space $\mathbb R^\infty$ and morphisms  local
diffeomorphisms (i.e. \'etale morphisms) of one manifold into
another one. In \cite{BR} it is shown that the usual smooth
technique can be extended naturally to such manifolds.

\item[4] The category $\mathcal M$ of smooth manifolds (of arbitrary finite
dimension) and morphisms smooth maps of one manifold into another
one.

\item[5] Let $G$ be a Lie group.  Consider a category $\mathcal
P_n(G)$ with objects smooth principal $G$-bundles with
$n$-dimensional bases whose morphisms are morphisms of such
principal $G$-bundles which are projected to the \'etale maps of
bases.

\end{itemize}

Let $\mathcal C$ be any of the categories above. For any
$a\in\mathcal C_0$, an $a$-sieve $S$ is called a cover, if $S$
containes the set of inclusions $a_i\to a$, where  $a_i$ is an open
cover of $a$. It is easy to see that these data define a
Grothendieck topology on $\mathcal C$.

Next we apply some general constructions of \cite{L6} to the
category $\mathcal C$.

Denote by $J$ the forgetful functor from $\mathcal C$ to the
category {\itshape Sets} of sets. Let $X$ be a set. A $\mathcal
C$-{\bfseries  chart} on $X$ is a map $k:c\to X$, where
$c\in\mathcal C_0$. Put $D(k)=c$. Given two $\mathcal C$-charts
$k_1$ and $k_2$, {\bfseries a morphism} from $k_1$ to $k_2$ is a
$\mathcal C$-morphism $f:D(k_1)\to D(k_2)$ such that $k_1= k_2\o f$.
Let $\mathcal C(X)$ be a category of $\mathcal C$-charts on $X$.  By
definition, $D:\mathcal C(X)\to\mathcal C$ is a covariant functor.
For a set $\Ph(X)$ of $\mathcal C$-charts on $X$, denote by
$\mathcal C_{\Ph(X)}$ the full subcategory of $\mathcal C(X)$ with
the object set $\Ph(X)$.

\begin{definition}\label{C-atlas}
A set $\Ph(X)$ of $\mathcal C$-charts on $X$ is called a $\mathcal
C$-atlas on $X$, if the set $X$ with the set of maps $k:J\o D(k)\to
X$ ($k\in\Ph(X))$ is an inductive limit $\varinjlim J\o D$ of the
functor $J\o D:\mathcal C_{\Ph(X)}\to Sets$.
\end{definition}

Let $\Ph(X)$ be a $\mathcal C$-atlas on $X$. Evidently the functor
$D:\mathcal C_{\Phi(X)}\to \mathcal C$ maps $\mathcal C_{\Ph(X)}$
onto some subcategory $\mathcal C(\Ph(X))$ of $\mathcal C$ and the
pair $(X,\Ph(X))$ can be considered as an inductive limit of the
restriction of $J$ to $\mathcal C(\Ph(X))$. Conversely, for any
subcategory $\mathcal B$ of $\mathcal C$, the inductive  limit of
the restriction of $J$ to $\mathcal B$ is an $\mathcal C$-atlas on
some set $X$. Thus one can define a $\mathcal C$-atlas as an
inductive limit of the restriction of $J$ to some subcategory of
$\mathcal C$. The set $X$ for this $\mathcal C$-atlas is determined
by this inductive limit up to bijective map. Note that, for every
set $\Ph(X)$ of $\mathcal C$-charts on $X$, by the definition of an
inductive limit, $\varinjlim J\o D$ is a set $\widetilde X$ given
with a set of maps $\widetilde k:J\o D(k)\to \widetilde X$ ($k\in
\Ph(X)$) such that the set of maps $\widetilde k:D(k)\to\tilde X$ is
a $\mathcal C$-atlas on $\widetilde X$ and there is a unique map
$h:\widetilde X\to X$ such that, for every $k\in \Phi(X)$, $k=h\o
\widetilde k$.

Let $\Ph$ be a set of $\mathcal C$-charts on $X$.  We shall say
that, for a $b\in \mathcal C_0$ a $\mathcal C$-chart on $X$ of the
type $k\o D(f)$, where $k\in \Ph$, $c=D(k)$, and $f:b\to c$ is a
morphism of $\mathcal C$, {\bfseries is obtained by restriction of
$k$ to $b$ with respect to} $f$. Denote by $R(\Ph)$ the set of
$\mathcal C$-charts on $X$  obtained by restriction of the charts
from $\Ph$ with respect to the all appropriate morphisms of
$\mathcal C$. Clearly, we have $\Ph\subset R(\Ph)$ and
$R^2(\Ph)=R(\Ph)$.

We shall say that a $\mathcal C$-chart $k$ on $X$ {\bfseries is
obtained by gluing from} $\Ph$, if there are a family $\{k_i\}$ of
$\mathcal C$-charts from $\Ph$ and the set of morphisms
$\{f_i:D(k_i)\to D(k)\}$ generating a cover of $D(k)$, such that
$k_i=k\o f_i$. Denote by $G(\Ph)$ the set of $\mathcal C$-charts on
$X$  obtained by gluing from $\Ph$. Clearly, we have $\Ph\subset
G(\Ph)$ and $G^2(\Ph)=G(\Ph)$. Put $\bar\Ph=G\o R(\Ph)$.

\begin{proposition} (\cite{L6}, Proposition 2.1.2)
The correspondence $\Ph\to \bar\Ph$ is a closure on the set of
$\mathcal M_n$-atlases on $X$, i.e. the following conditions hold:
\begin{enumerate}
\item $\Ph\subset\bar\Ph$;
\item If $\Ph_1\subset\Ph_2$, then $\bar\Ph_1\subset\bar\Ph_2$;
\item $\Bar{\Bar\Ph}=\bar\Ph$.
\end{enumerate}
\end{proposition}
Evidently, for a $\mathcal C$-atlas on $X$, $\bar\Ph$ is a $\mathcal
C$-atlas on $X$. A $\mathcal C$-atlas $\Ph$ on $X$ is called
{\bfseries closed}\footnote{A closed $\mathcal C$-atlas is the same
as a {\bf maximal} $\mathcal C$-atlas. A $\mathcal C$-atlas $\Ph$ is
called {\bf full} if for any $k\in\bar\Ph$ and any $x\in D(x)$,
there exist a chart $k_1\in\Ph$ and a morphism $m:k_1\to k$ such
that $k_1(y)=x$ for some $y\in D(k_1)$. The notion of a full atlas
is similar to the notion from \cite{Cr} of a transversal basis for a
foliation.} if $\Bar\Ph=\Ph$.

\begin{definition} A pair $(X,\Ph)$, where $\Ph$ is a closed $\mathcal C$-atlas on
$X$, is called $\mathcal C$-{\bfseries space}. Two $\mathcal
C$-atlases $\Ph_1$ and $\Ph_2$ on $X$ are called {\bfseries
equivalent} if $\bar\Ph_1=\bar\Ph_2$.\end{definition}

By definition, equivalent $\mathcal C$-atlases on $X$ determine the
same structure of $\mathcal C$-space on $X$. It is evident that two
$\mathcal C$-atlases $\Ph_1$ and $\Ph_2$ on $X$ are equivalent iff
$R(\Ph_1)=R(\Ph_2)$.

\begin{definition}\label{morphism}
Let $(X_1,\Ph_1)$ and $(X_2,\Ph_2)$ be two $\mathcal C$-spaces. A
morphism $(X_1,\Ph_1)\to (X_2,\Ph_2)$ is a map $f:X_1\to X_2$ , if,
for each $k\in\Ph_1$, we have $f\o k\in\Ph_2$.
\end{definition}

Thus, we have a category $\mathcal C_{\on{sp}}$ of $\mathcal
C$-spaces, in particular, we have the categories $\mathcal
M_{n,\on{sp}}$, $\mathcal D_{n,\on{sp}}$, $\mathcal
M_{\infty,\on{sp}}$, $\mathcal M_{\on{sp}}$, and $\mathcal
P(G)_{n,\on{sp}}$.

For $a\in\mathcal C_0$, the identity map $\on{id}:a\to a$ is an
$\mathcal C$-chart on $a$ and $\Ph(a)=\{\on{id}\}$ is a $\mathcal
C$-atlas on $a$. Therefore, one can consider $a$ as an object of the
category $\mathcal C_{\on{sp}}$ and the corresponding map $\mathcal
C\to\mathcal C_{\on{sp}}$ is an isomorphism of the category
$\mathcal C$ onto a full subcategory of the category $\mathcal
C_{\on{sp}}$.

It is evident that each object of the category $\mathcal M_n$ is an
object of the category $\mathcal D_{n,\on{sp}}$. Since the category
$\mathcal D_n$ is a full subcategory of $\mathcal M_n$, the
corresponding natural map from $\mathcal D_{n,\on{sp}}$ into
$\mathcal M_{n,\on{sp}}$ is a category isomorphism (see \cite{L6},
Theorem 2.3.4). Nevertheless, later it is useful to consider both of
these categories. Moreover, each $\mathcal M_n$-atlas $\Ph$ on $X$
is a $\mathcal M$-atlas on $X$ (with the larger set of morphisms of
charts) and equivalent $\mathcal M_n$-atlases on $X$ are equivalent
as $\mathcal M$-atlases on $X$.

It is clear that the definition of a $\mathcal D_n$-space is a
direct generalization of the definition of an $n$-dimensional
orbifold. See also the similar notions of $S$-atlas of Van Est
\cite{VE} and $QF$-variety of Pradines, Wouafo-Kamga \cite{PW}. The
category of $\mathcal D$-spaces coincides with the category of
diffeological spaces introduced by Souriau~\cite{So}.

\subsection{Structures on $\mathcal{M}_n$-spaces}\label{secstruc}

Now we show that an $\mathcal M_n$-space has a rich smooth
structure.

Let $F:\mathcal M_n\to Sets$ be a covariant functor. Define an
extension of $F$ to the category $\mathcal M_{n,\on{sp}}$ as
follows. Let $X$ be a $\mathcal M_n$-space and let $\Ph$ be a
$\mathcal M_n$-atlas on $X$. By definition of inductive limit,
$\varinjlim F\o D$ is a set $\tilde X$ with the set of maps $\tilde
\Ph$: $\tilde k:F\o D(k)\to\tilde X$ ($k\in\Ph$) compatible with the
morphisms of the category $\mathcal C_\Ph$. Put  $F(\Ph)=\tilde X$.
It is easy to check that, if two $\mathcal M_n$-atlases  $\Ph_1$ and
$\Ph_2$ on $X$ are equivalent, there is a natural bijective map
$F(\Ph_1)\to F(\Ph_2)$. We extend thus the functor $F$ to the
category $\mathcal M_{n,\on{sp}}$. It is easy to see that this
extension is compatible with functor morphisms. For example, let
$C(M)$ be the set of smooth singular simplexes on a manifold $M$.
Evidently $C(M)$ is a covariant functor $\mathcal M_n\to Sets$ which
has a natural extension to the category $\mathcal M_{n,\on{sp}}$.
Then we have natural notion of the complex of finite smooth singular
chains on each $\mathcal M_n$-space.

Let $F:\mathcal M_n\to\mathcal M_N$ ($F:\mathcal M_n\to\mathcal
M_\infty$) be a covariant functor and let $\Ph$ be a $\mathcal
M_n$-atlas on $X$. Put $\tilde F=J\o F$ and $\tilde X=\varinjlim
\tilde F\o D$. Then the set of maps $\tilde \Ph$: $\tilde k:\tilde
F\o D(k)\to\tilde X$ ($k\in\Ph$) is an $\mathcal M_N$-atlas
($\mathcal M_\infty$-atlas) on the set $\tilde X=F(\Ph)$.  We obtain
thus the covariant functor $F:\mathcal M_{n,\on{sp}}\to \mathcal
M_{N,\on{sp}}$ ($F:\mathcal M_{n,\on{sp}}\to \mathcal
M_{\infty,\on{sp}}$) which is an extension of the initial functor
and is compatible with functor morphisms. We have thus the notions
of the tangent bundle, the frame bundle, and so on for any $\mathcal
M_n$-space.

For example, put $F(M)=\on{T}(M)$ or $F(M)=\on{Fr}(M)$, where
$\on{T}(M)$ is a tangent bundle of $M$ and $\on{Fr}(M)$ is the frame
bundle of $M$. Applying the above extension to these functors we
obtain the tangent bundle for any $\mathcal M$-space and frame
bundle  for any $\mathcal M_n$-space. Moreover, the projection
$\on{T}(M)\to M$ ($\on{Fr}(M)\to M$) is functor morphism from the
category $\mathcal M_{\on{sp}}$ to the category $\mathcal
M_{\on{sp}}$ (from the category of $\mathcal M_{n,\on{sp}}$ to
$\mathcal M_{n(n+1),\on{sp}}$). Then we have a projection of the
tangent bundle of a $\mathcal M$-space $X$ (the frame bundle of a
$\mathcal M_n$-space) to this  $\mathcal M$-space). Note that the
fibers of the projection $T(X)\to X$ are not vector spaces in
general. By definition, a tangent vector $\ta\in T(X,\Ph)$ is a
family of tangent vectors $\ta_k$ to $M$ for each $k\in\Ph$ such
that these vectors are compatible with the morphisms of $\mathcal
C_{\Ph}$.

Assume that a category $\mathcal C$ has products and the kernel of
each pair of morphisms exists. For example, $\mathcal C$ is a
category of sets, groups, rings or algebras. Let $F:\mathcal
M_n\to\mathcal C$ be a contravariant functor.  Define an extension
of $F$ to the category of $\mathcal M_n$-spaces as follows. Let
$\Ph$ be a $\mathcal M_n$-atlas on $X$. By the definition of
projective limit, is an object $\hat X$ of the category $\mathcal C$
with the set of morphisms $\hat\Ph$: $\hat k:\hat X\to F\o D(k)$
($k\in\Ph$) compatible with the morphisms of the category $\mathcal
C_\Ph$. It is known that the projective limit $\varprojlim F\o D$
exists. It is easy to see that, if two $\mathcal M_n$-atlases
$\Ph_1$ and $\Ph_2$ on $X$ are equivalent, there is a natural
isomorphism $\hat\Ph_1\to \hat\Ph_2$. Put $F(\Ph)=\hat\Ph$. It is
evident that $F$ is an extension of the initial functor to the
category $\mathcal M_{n,\on{sp}}$  which is compatible with functor
morphisms.  If $F(M)$ is a sheaf on each $n$-dimensional manifold
$M$, $F(X,\Ph)$ is called {\bfseries a sheaf on a $\mathcal
M_n$-space} $(X,\Ph)$. The examples of sheaves are the tensor
algebra $\Theta(M)$ of tensor fields, the de Rham complex $\Om^*(M)$
on a manifold $M$, and so on.

Put $F(M)=\Om^*(M)$, where $\Om^*(M)$ is the graded differential
algebra of differential forms on $M$. Then  $\Om^*(X)$ is a graded
differential algebra for any $\mathcal M_n$-space $X$. By
definition, a $p$-differential form on $X$ is a family of
$p$-differential forms $\om(D(k))$ on $D(k)$ for each $k\in\Ph)$
such that these forms are compatible with the morphisms of $\mathcal
C_{\Ph}$. It is easy to see that the notion of integral of a
differential $p$-form over a smooth singular $p$-chain is extended
naturally to the category $\mathcal M_{n,\on{sp}}$ and the
corresponding Stokes theorem is true.  Note that in the foliation
theory the corresponding  differential algebra $\Om^*(X)$ is called
the algebra of basic differential forms and the cohomology $H^*(X)$
of $\Om^*(X)$ is called the basic cohomology of foliation.

Thus, we see that the elements of the smooth techniques on smooth
manifolds which have a functorial nature could be extended to the
category $\mathcal M_{n,\on{sp}}$.

\begin{remark*}
To extend the functors above it is possible to use, instead of a
maximal $\mathcal M_n$-atlas on the set $X$, any full $\mathcal
M_n$-atlas on the set $X$. The result will be the same. For example,
one can use a maximal $\mathcal D_n$-atlas on $X$.
\end{remark*}

\subsection{Groupoids}

Recall the definition of smooth and \'etale groupoids.

{\bfseries  A groupoid}  is a category $\mathcal G$ all whose
morphisms  are isomorphisms. Thus, we have a set $\mathcal{G}_0$ of
objects $x,y,...$ and the set $\mathcal{G}_1$ of morphisms
$f,g,...$. Each morphism $f:x\to y$ has a source $s(f)=x$ and a
target $t(f)=y$. Two morphisms $f$ and $g$ with $s(f)=t(g)$ can be
composed as $fg:s(g)\to t(f)$. This composition is associative, has
a unit $1_x:x\to x$ for each $x\in\mathcal{G}_0$, and has an inverse
$i(f)=f^{-1}:t(f)\to s(f)$ for each $f\in\mathcal{G}_1$. Put
$$\mathcal{G}_2=\mathcal{G}_1\times_{\mathcal{G}_0}\mathcal{G}_1=\{(f,g)\in
\mathcal{G}_1\x\mathcal{G}_1: s(f)=t(g)\}.$$ All the structure is
contained in the diagram
\begin{equation}\label{groupoid}
\mathcal{G}_2\stackrel{m}{\longrightarrow}\mathcal{G}_1\stackrel{i}{\longrightarrow}
\mathcal{G}_1\stackbin[t]{s}{\rightrightarrows}\mathcal{G}_0\stackrel{u}{\longrightarrow}\mathcal{G}_1.
\end{equation}
Here $s$ and $t$ are the source and the target, $m$ denotes
composition, $i$ is the  inverse, and $u(x)=1_x$. The groupoid is
{\bfseries smooth} if $\mathcal{G}_0$ and $\mathcal{G}_1$ are smooth
manifolds ($\mathcal{G}_1$ is non-Hausdorff, in general), all the
structure maps in \eqref{groupoid} are smooth,  and $s$ and $t$ are
submersions, so that $\mathcal{G}_2$ is a smooth manifold as well. A
smooth groupoid is {\bfseries \'etale} if all the structure maps in
\eqref{groupoid} are \'etale maps (it suffices to require this
for~$s$).

Let $\mathcal C$ be a category. Define an equivalence relation $x
\sim y$ on $\mathcal C_0$ generated by the relations of the
following type: there is a morphism $f$ of the category $\mathcal C$
such that either $f:x\to y$, or $f:y\to x$. Thus, the quotient set
$\mathcal C_0/\sim$ equals $\pi_0(\mathcal C)$, the set of connected
components of $\mathcal C$. It is evident that, if $\mathcal
C=\mathcal G$ is a groupoid, $x\sim y$ iff there is a morphism
$f:x\to y$ of the groupoid $\mathcal G$.

Let $\mathcal C=\mathcal M_n, \mathcal D_n, \mathcal M$. For each
$\mathcal C$-atlas $\Ph$ on $X$, we define a category $\mathcal
G(\Ph)$ as follows. The object space $\mathcal G(\Ph)_0$ is the
disjoint union of the domains $D(k)$ of all $\mathcal C$-charts
$k\in\Ph$. For any $x_1,x_2\in\mathcal G(\Ph)_0$, a morphism $x_1\to
x_2$ is the germ of a morphism $f:k_1\to k_2$ at $x_1$, where $k_1,
k_2\in\Ph$, $x_1\in D(k_1)$, $x_2\in D(k_2)$ and $f(x_1)=x_2$.  It
is clear that the natural map $\mathcal G(\Ph)_0\to X$ induces a
bijective map $(\mathcal G(\Ph)/\sim)\to X$. Note that, by
definition, $\mathcal G(\Ph)_0$ is an $n$-dimensional manifold for
$\mathcal C=\mathcal M_n,\mathcal D_n$ and a disjoint union of
manifolds for $\mathcal C=\mathcal M$.

Conversely, consider the category $\mathcal G(\Ph)$ and the
projection $\pi:\mathcal G(\Ph)_0\to X=\mathcal G(\Ph)/\sim$. Let
$\Ph'$ be the set of restrictions of $\pi$ to all open subsets of
$\mathcal G(\Ph)_0$. It is easy to check that $\Ph'$ is a $\mathcal
C$-atlas on $X$,  $\Ph\subset\Ph'$, and $\mathcal C$-atlases $\Ph$
and $\Ph'$ are equivalent.

A $\mathcal C$-atlas $\Ph$ on $X$ is called {\bfseries locally
invertible} if the category $\mathcal G(\Ph)$ is a groupoid, i.e.,
if for any $k_1,k_2\in\Ph$, each $\mathcal C_\Ph$-morphism $f:k_1\to
k_2$, and each $x\in D(k_1)$, there are $k_3\in\Ph$ and $\mathcal
C_\Ph$-morphisms  $i:k_3\to k_2$, $g:k_3\to k_1$ such that $f(x)\in
i(D(k_3))$, $fg=i$, and $i$ maps $D(k_3)$ diffeomorphically onto an
open subset of $D(k_2)$. In particular, an $\mathcal M_n$-atlas
$\Ph$ on $X$ is locally invertible if, for each $k\in\Ph$ and each
open subset $U$ of $D(k)$, we have $(U,k|_U)\in\Ph$.

Let $\Ph$ be a locally invertible $\mathcal M_n$-atlas on $X$. The
set $\mathcal G_0(\Ph)$  is a manifold by definition. Endow the set
of morphisms $\mathcal G_1(\Ph(X))$ with  a topology generated by
sets of germs of morphisms $f:k_1\to k_2$ of the category $C_\Ph$ at
points $x\in D(k_1)$. Note that the manifold $\mathcal G_1(\Ph)$ may
be non-Hausdorff. Since the sourse map $s:\mathcal
G(\Ph)_1\to\mathcal G(\Ph)_0$ is \'etale, $\mathcal G(\Ph)$ is an
\'etale groupoid.

If $\mathcal G$ is an \'etale groupoid, each morphism $g:x\to y$ in
$\mathcal G$ uniquely determines the germ of a diffeomorphism
$\tilde g:(x,\mathcal G_0)\to(y,\mathcal G_0)$, namely $\tilde g$ is
a germ at x of the local diffeomorphism $t\o\si$ of $\mathcal G_0$,
where $\si$ is a section of $s:\mathcal G_1\to\mathcal G_0$ on a
neighborhood $U$ of $x$ , with $\si(x)=y$ and $U$ is so small that
$t\o \si$ is a diffeomorphism from $U$ onto its image. In
particular, this construction gives a group homomorphism $\mathcal
G_x\to\on{Diff}_x(\mathcal G_0)$. The \'etale groupoid is called
{\bfseries effective} (or $S$-atlas in sence of van Est \cite{VE})
whenever this homomorphism is injective for each $x\in\mathcal G_0$.
By definition, the groupoid $\mathcal G(\Ph)$ is effective.

\subsection{Examples of $\mathcal{M}_n$-spaces}

Now we give examples of $\mathcal M_n$-spaces.

1. Let $N$ be an $(m+n)$-dimensional manifold and let $\mathcal F$
be a foliation  of codimension $n$ on $N$. Denote by $L(\mathcal F)$
the set of leaves of $\mathcal F$ and by $\Ph(\mathcal F)$ the set
of smooth maps $M\to N$ transverse to the leaves of $\mathcal F$,
where $M\in (\mathcal M_n)_0$.

\begin{theorem}(\cite{L5},\cite{L6})
The set $\Ph(\mathcal F)$ of $\mathcal M_n$-charts on  $L(\mathcal
F)$  is a locally invertible $\mathcal M_n$-atlas on $L(\mathcal
F)$.
\end{theorem}

Let $T$ be a complete transversal of the foliation $\mathcal F$,
i.e. an embedded $n$-submanifold $M\subset N$ which transverse to
the leaves of $\mathcal F$ hitting each leaf at least once. Let
$p:T\to L(\mathcal F)$ be the projection. It is easy to check that
the set $\Ph(T)$ of the restrictions of $p$ to open submanifolds of
$T$ is a locally invertible $\mathcal M_n$-atlas on $L(\mathcal F)$
and $\Ph(T)$ is equivalent to the $\mathcal M_n$-atlas $\Ph(\mathcal
F)$. By definition, the set of objects for the  groupoid $\mathcal
G(\Ph(T))$ equals $T$. Evidently, the  groupoid $\mathcal G(\Ph(T))$
coincides with the so-called \'etale model $\on{Hol}_T(N,\mathcal
F)$ of the holonomy groupoid $\on{Hol}(N,\mathcal F)$ of the
foliated manifold $(N,\mathcal F)$ (see, for example, \cite{Mo}).

\begin{remark}
For study the the leaf space of a foliated manifold $(N,\mathcal F)$
one usually uses the holonomy groupoid of $\mathcal F$ with $N$ as
the set of objects. From the point of view accepted in this paper
this is not natural since different foliated manifolds may have
isomorphic leaf spaces.
\end{remark}

2. Let $N$ be a manifold of dimension $n$ and let $\Ga$ be a
pseudogroup of local diffeomorphisms\footnote{Here a local
diffeomorphism of a manifold $N$ is a smooth map $f:U\to N$ defined
on an open subset $U\subset N$ that maps $U$ diffeomorphically onto
$f(U)$. Compare this notion of a local diffeomorphism with the one
from Section \ref{secdefCsp}} of $N$, in particular, a group of
diffeomorphisms of $N$. Let $N/\Ga$ be the orbit space of $\Ga$ and
let $p:N\to N/\Ga$ be the natural projection. Consider the set
$\Ph(\Ga)$ of $\mathcal M_n$-charts on $N/\Ga$ of the type
$(M',p|_{M'})$, where $M'$ is  an open subset of $N$.

\begin{theorem}\label{pseudogroup}(\cite{L5},\cite{L6})
The set $\Ph(\Ga)$ of $\mathcal M_n$-charts on $N/\Ga$ is a locally
invertible $\mathcal M_n$-atlas on $N/\Ga$.
\end{theorem}

3. Let $X$ be a one point set $\{pt\}$ and let the set of $\mathcal
M_n$-charts be the set $\Ph(\{pt\})$ of all maps $M\to\{pt\}$, where
$M\in {\mathcal M_n}_0$. It is evident that the pair
$(\{pt\},\Ph(\{pt\})$ is a final object of the category $\mathcal
M_{n,\on{sp}}$.

4. Consider the space $\mathbb R^n$ and the pseudogroup $\Ga(\mathbf
R^n)$ of all local diffeomorphisms of $\mathbb R^n$. It is clear
that $\mathbb R^n/\Ga(\mathbb R^n)$ is an one point set $\{pt\}$
and, by Theorem~\ref{pseudogroup}, the set $\Ph(\Ga(\mathbb R^n))$
of $\mathcal M_n$-charts on $\{pt\}$, is a locally invertible
$\mathcal M_n$-atlas on $\{pt\}$. Then the corresponding $\mathcal
M_n$-space is a final object of the category of $\mathcal M_n$
spaces. By definition, the groupoid corresponding to
$\Ph(\Ga(\mathbb R^n))$ is the classifying Haefliger groupoid
$\Ga_n$ \cite{Hf1}. To get an equivalent classifying space one can
take instead of $\mathbb R^n$ an arbitrary connected $n$-dimensional
manifold $M$  and its pseudogroup $\Ga(M)$ of all local
diffeomorphisms, this is obvious, since $M /\Ga(M)$ is again a
one-point set.

\section{\'Etale structures}

Consider the following generalizations of an $\mathcal M_n$-space.

\begin{definition}\label{E-atlas}
Let $\Ph$ be an $\mathcal M_n$-atlas on a set $X$ and let $\mathcal
B_\Ph$ be a subcategory of the category $\mathcal C_\Ph$ with the
same set of objects $\Ph$. An \'etale structure on $X$ is a category
$\mathcal E_\Ph$ such that $(\mathcal E_\Ph)_0=\Ph$ given with a
full covariant functor $\Pi:\mathcal E_\Ph\to\mathcal B_\Ph$ which
is identical on the set of objects satisfying the following
condition: the set of maps $k:J\o D(k)\to X$ ($k\in\Ph)$ is an
inductive limit $\varinjlim J\o D\o\Pi$ of the functor $J\o D\o\Pi:
\mathcal E_\Ph\to Sets$. An \'etale structure $\mathcal E_\Ph$ on
$X$ is called effective if $\mathcal E_\Ph=\mathcal B_\Ph$. An
\'etale structure $\mathcal E_\Ph$ on $X$ is called locally
invertible if for any $k_1,k_2\in\Ph$, each $\mathcal
E_\Ph$-morphism $f:k_1\to k_2$, and each $x\in D(k_1)$, there are
$k_3\in\Ph$ and $\mathcal E_\Ph$-morphisms  $i:k_3\to k_2$,
$g:k_3\to k_1$ such that $f(x)\in i(D(k_3))$, $fg=i$, and $i$ maps
$D(k_3)$ diffeomorphically onto an open subset of $D(k_2)$.
\end{definition}

It is clear that each \'etale structure $\mathcal E_\Ph$ on $X$
defines uniquely the corresponding structure of $\mathcal M_n$-space
on $X$ and, therefore, the corresponding structure of diffeological
space.

For example, let $\Ga$ be a pseudogroup of local diffeomorphisms of
an $n$-dimensional manifold $M$, $\Ph(\Ga)$ consists of the
restrictions of the projection $M\to M/\Ga=X$ to the domains of the
local diffeomorphisms from $\Ga$, and morphisms of $\mathcal
E_{\Ph(\Ga)}$  are local diffeomorphisms from $\Ga$. By definition,
$\mathcal E_{\Ph(\Ga)}$ is an effective  locally invertible \'etale
structure on $X$.

Let $\mathcal E_{\Ph}$ be an effective  \'etale structure on $X$
with the set of objects $\Ph$. For each $M\in\mathcal M_0$,
$k\in\Ph$, and a smooth map $f:M\to D(k)$, the composition $k\o f$
is an $\mathcal M$-chart on $X$.

\begin{definition}\label{dover}
 A diffeological structure on the set $X$ over an effective  inverible \'etale structure $\mathcal E_{\Ph}$ on $X$  is a category
$\mathcal E_{\Ph,d}$ whose objects are the pairs $(k,f)$, where
$k\in\Ph$ and  $f:M\to D(k)$ is a smooth map for some $M\in\mathcal
M_0$, and, for any objects $(k_1,f_1)$, $(k_2,f_2)$, a morphism
$(k_1,f_1)\to (k_2,f_2)$ is defined whenever $f_1$ and $f_2$ have
the same domain $M$ and in this case is a morphism $g:k_1\to k_2$
belonging to $(\mathcal E_{\Ph})_1$ such that $f_2=g\o f_1$.
\end{definition}

We generalize definition \ref{morphism} as follows.
\begin{definition}\label{morphism1}
Let $\mathcal E_{\Ph_1}$ and $\mathcal E_{\Ph_2}$ be two \'etale
structures on the sets $X_1$ and $X_2$ respectively. A morphism of
$f:\mathcal E_{\Ph_1}\to\mathcal E_{\Ph_2}$ of \'etale structures is
a  map $f:X_1\to X_2$ that satisfy following conditions:
\begin{enumerate}
\item  For each $k\in\Ph_1$, we have $f\o k\in\Ph_2$;
\item For each morphism $m:k_1\to k_2$ from $\mathcal E_{\Ph_1}$, the map $m:D(f\o k_1)=D(k_1)\to D(k_2)=D(f\o k_2)$ is a morphism
$f\o k_1\to f\o k_2$ from $\mathcal E_{\Ph_2}$.
\end{enumerate}
\end{definition}

Let $\mathcal E_\Ph$ be an \'etale structure on a set $X$. Define a
category $\mathcal G(\mathcal E_\Ph)$ as follows. The set of objects
$\mathcal G(\mathcal E_\Ph)_0$  equals the disjoint union of domains
of $\mathcal M_n$-charts from $\Ph$. For $x,y\in\mathcal G(\mathcal
E_\Ph)_0$ and $k_1,k_2\in\Ph$,  a morphism $x\to y$ is a germ at $x$
of a $\mathcal E_\Ph$-morphism $g:k_1\to k_2$ such that $x\in
D(k_1)$, $y\in D(k_2)$, and $D(g)(x)=y$. If an \'etale structure
$\mathcal E_\Ph$ on $X$ is locally invertible, the category
$\mathcal G(\mathcal E_\Ph)$ is a groupoid.

By definition, the set of objects $\mathcal G(\mathcal E_\Ph)_0$ of
$\mathcal G(\mathcal E_\Ph)$ is an $n$-dimensional manifold. Endow
the set of morphisms $\mathcal G(\mathcal E_\Ph)_1$ of $\mathcal
G(\mathcal E_\Ph)$ with  a topology generated by the sets of germs
of morphisms $f:k_1\to k_2$ of the category $\mathcal E_\Ph$ at
points $x\in D(k_1)$. Since such sets are $n$-dimensional manifolds,
$\mathcal G(\mathcal E_\Ph)_1$ is an $n$-dimensional manifold as
well. By construction, for a locally invertible \'etale structure
$\mathcal E_\Ph$ the category $\mathcal G(\mathcal E_\Ph)$ is an
\'etale groupoid.

\begin{remark}\label{reduced}
Note that, in the construction above, the category $\mathcal
G(\mathcal E_\Ph)$   very often can be reduced essentially. Assume
that we have a surjective  \'etale map $\mathcal G(\mathcal
E_\Ph)_0\to M$, where $M$ is an $n$-dimensional manifold. Then one
can take for  $\mathcal G(\mathcal E_\Ph)_0$ the manifold $M$ and,
for $x,y\in M$, take for a $\mathcal E$-morphism $x\to y$ the germ
of local diffeomorphism $(M,x)\to (M,y)$  which can be lifted to a
germ of some $\mathcal G(\mathcal E_\Ph)$-morphism $k_1\to k_2$. By
definition, the reduced category has the same quotient space
$X=\mathcal G(\Ph)/\sim$. For example, let $\Ga$ be a pseudogroup of
local diffeomorphisms of an $n$-dimensional manifold $M$ and let
$\mathcal E_{\Ph(\Ga)}$ be the corresponding locally invertible
\'etale structure on $X=M/\Ga$. By the standard procedure, the set
of objects $\mathcal G(\mathcal E_{\Ph(\Ga)})_0$ is the disjoint
union of domains of local diffeomorphisms from $\Ga$. But there is
the natural \'etale projection $\mathcal G(\mathcal
E_{\Ph(\Ga)})_0\to M$. Therefore, one can take for the object set of
the reduced category $\mathcal G(\mathcal E_{\Ph(\Ga)})$ the
manifold $M$.
\end{remark}

Let $\mathcal G$ be a smooth groupoid and $X=\mathcal G_0/\sim$.
Since $s$ is a submersion, for each morphism $g:x\to y$ from
$\mathcal G_1$, there is an open neighborhood $U$ of $x\in\mathcal
G_0$ and a smooth section $\si:U\to\mathcal G_1$ of the map $s$ such
that $\si(x)=g$. Put $f_\si=t\o\si:U\to\mathcal G_0$.

Define a category $\mathcal E(\mathcal G)$ as follows. The objects
of this category are open subsets of $\mathcal G_0$.  For
$U,V\in\mathcal E(\mathcal G)_0$, a morphism from $U$ to $V$ in the
category $\mathcal E(\mathcal G)$ is a section $\si:U\to\mathcal
G_1$ of the source map $s$ with the properties that $f_\si$ is an
\'etale map and $f_\si(U)\subset V$. The composition of morphisms
$\si:U\to\mathcal G_1$ and $\ta:V\to\mathcal G_1$, where
$f_\si(U)\subset V$, is defined by
$(\ta\si)(x)=\ta(f_\si(x))\cdot\si(x)$ ($x\in U$), where the
multiplication is taken in $\mathcal G$. Note that the category
$\mathcal E(\mathcal G)$ is similar to the category
$\on{Emb}(\mathcal G)$ defined in \cite{Mo1}.

Let $p:\mathcal G_0\to X$ be the projection. For each open subset
$U$ of $\mathcal G_0$, denote by $k_U$ the restriction of $p$ to
$U$. By definition, $k_U$ is a $\mathcal M_n$-chart on $X$. For two
open subsets $U,V$ of $\mathcal G_0$ and the section $\si$ of the
source map $s$ such that $f_\si$ is \'etale and $f_\si(U)\subset V$,
one can consider $f_\si$ as a morphism $k_U\to k_V$ of $\mathcal
M_n$-charts on $X$. Denote by $\Ph=\Ph(\mathcal G)$ the set of
$\mathcal M_n$-charts on $X$ of type $k_U$. Consider also the
category $\mathcal B(\mathcal G)$ with $\Ph$ as the set of objects
whose morphisms are the \'etale maps $f_\si:U\to V$ as above. We
have the natural covariant functor $\Pi:\mathcal E(\mathcal
G)\to\mathcal B(\mathcal G)$ such that, for an open subset
$U\subset\mathcal G_0$, we have  $\Pi(U)=U$ and, for a section $\si$
as above, we have $\Pi(\si)=f_\si:k_U\to k_V$. By definition, the
functor $\Pi$ is full.

Consider the covariant functor $J\o D\o\Pi:\mathcal E(\mathcal G)\to
Sets$. Prove that the set of maps $k_U:J\o D(k_U):U\to X$ is an
inductive limit $\varinjlim J\o D$ of the functor $J\o D\o\Pi$. It
is obvious that the codomains of the $\mathcal M_n$-charts $k_U$
cover $X$. Let $U,V\in\mathcal E(\mathcal G)_0$, $x\in U$, $y\in V$,
and $k_U(x)=k_V(y)$. Since $X=\mathcal G_0/\sim$, there is a
$\mathcal G$-morphism $g:x\to y$. First assume that $x=y$ and
$U\subset V$. Consider the section $\ve:U\to\mathcal G_1$ defined by
$\ve(x')=1_{x'}$ for any $x'\in U$. Then $\ve$ is a morphism $k_U\to
k_V$ such that $f_\ve(x)=y$ and we are done. In the general case
there exists\footnote{The section $\si$ exists by the following
reason. Since both maps $ds$ and $dt$ are epimorphisms, there exists
a vector subspace $L\subset T_g {\mathcal G} _1$ such that $ds: L\to
T_x{\mathcal G} _0$ and $dt:L\to T_y {\mathcal G} _0$ are
isomorphisms. It is possible to chose a submanifold $\Sigma \subset
{\mathcal G}_1$   passing through the point $g$ such that
$T_g\Sigma=L$ and  $s|_\Sigma,t|_\Sigma :\Sigma \to {\mathcal G}_0$
are diffeomorphisms onto their images. Then $f=(s|_\Sigma)^{-1}$ is
the required section. This remark is communicated to A.\,Galaev  by
Alexei\,Kotov.} a smooth section $\si:U\to\mathcal G_1$ of the
source map $s$ such that $f_\si$ is \'etale and $\si(x)=g$ (by the
argument above one can assume that $f_\si(U)\subset V$). Then for
$f_\si$ as a morphism $k_U\to k_V$, we have $f_\si(x)=y$ and we are
done again. Thus, $\mathcal E(\mathcal G)$ is a locally invertible
\'etale structure on $X$.

Let $\mathcal G$ be a smooth groupoid and let $\mathcal E(\mathcal
G)$ be the corresponding locally invertible \'etale structure on
$X=\mathcal G/\sim$. Consider the groupoid $\mathcal G(\mathcal
E(\mathcal G))$. By definition we have an obvious \'etale map
$\mathcal G(\mathcal E(\mathcal G))_0\to\mathcal G_0$. Then, by
remark \ref{reduced}, one can consider the corresponding reduced
groupoid $\mathcal G(\mathcal E(\mathcal G))$ with $\mathcal G_0$ as
the set of objects.

By construction, for any \'etale groupoid $\mathcal G$, we have
$\mathcal G(\mathcal E(\mathcal G))=\mathcal G$. Thus, to study an
\'etale groupoid $\mathcal G$ one can equivalently use the
corresponding \'etale structure~$\mathcal E_{\mathcal G,\Ph_S}$.

\section{Morita equivalence}

Let $\mathcal G$ and $\mathcal H$ be two smooth groupoids. A
covariant functor $F:\mathcal G\to\mathcal H$ is smooth if the
corresponding maps $\mathcal G_0\to\mathcal H_0$ and $\mathcal
G_1\to\mathcal H_1$ are smooth. It is clear that the morphism
$\mathcal E_{\Ph_1}\to\mathcal E_{\Ph_2}$ of locally invertible
\'etale structures on the sets $X_1$ and $X_2$ respectively induces
a smooth covariant functor $\mathcal G(\mathcal
E_{\Ph_1})\to\mathcal G(\mathcal E_{\Ph_2})$ of the corresponding
\'etale groupoids.

Recall some definitions.

Let $\mathcal G$ be a smooth groupoid and let $M$ be a smooth
manifold. {\bfseries  A right action} of $\mathcal G$ on $M$
consists of two smooth maps $\pi: M\to\mathcal G_0$ (the moment map)
and $m:M\x_{\mathcal G_0}\mathcal G_1=\{(x,g):\pi(x)=t(g)\}\to M$
(action) such that, denoting $m(x,g)=xg$, we have
$$
(xg)h=x(gh), \quad x1=x,\quad \pi(xg)=s(g).
$$
We will call $M$ a right $\mathcal G$-space with the moment map
$\pi$. A left $\mathcal G$-space is defined similarly.

A (right) $\mathcal G$-{\bfseries bundle} over the manifold $B$
consists of a smooth right $\mathcal G$-space $E$ and a smooth map
$p:E\to B$ which is $\mathcal G$-equivariant (i.e. $p(xg)=p(x)$). It
is called principal, if $p$ is a surjective submersion and the map
$E\x_{\mathcal G_0}\mathcal G_1\to E\x_BE$, defined by $(e,g)\to
(e,eg)$, is a diffeomorphism.

Let $\mathcal G$ and $\mathcal H$ be two smooth groupoids.
\begin{definition}\label{Skandalis}
A  homomorphism of groupoids (or the Hilsum-Skandalis map) $\mathcal
G\to\mathcal H$ (\cite {Hf},\cite{Mo},  \cite{Mr}) consists of a
manifold $P$, smooth maps (source and target): $s_P:P\to\mathcal
G_0$, $t_P: P\to\mathcal H_0$, a left action of $\mathcal G$ on $P$
with the moment map $s_P$, a right action of $\mathcal H$ on $P$
with the moment map $t_P$, such that
\begin{enumerate}
\item $s_P$ is $\mathcal H$-equivariant, $t_P$ is $\mathcal G$-equivariant;
\item the actions of $\mathcal G$ and $\mathcal H$ commute: $(gp)h=g(ph)$;
\item $s_P:P\to\mathcal G_0$, as an $\mathcal H$-bundle with the moment map $t_P$, is
principal.
\end{enumerate}
\end{definition}
The composition of two homomorphisms $P:\mathcal G\to\mathcal H$ and
$Q: \mathcal H\to \mathcal K$ is defined by dividing $P\x_{\mathcal
H_0}Q$ by the action of $\mathcal H$: $(p,q)h=(ph,h^{-1}q)$, and
taking the natural actions of $\mathcal G$ and $\mathcal K$.

Thus, we have the category of smooth groupoids. Two smooth groupoids
$\mathcal G$ and $\mathcal H$ are called {\bfseries Morita
equivalent} if they are isomorphic in the category of smooth
groupoids and an isomorphism $\mathcal G\to\mathcal H$ is called
{\bfseries a Morita equivalence}.

Now we compare definition \ref{morphism1} with the definition of a
morphism of smooth groupoids.

\begin{theorem}\label{homomorphism}
Let $\mathcal E_{\Ph_1}$ and $\mathcal E_{\Ph_2}$ be locally
invertible \'etale structures on the sets $X$ and $Y$ respectively
and let $\mathcal G=\mathcal G(\mathcal E_{\Ph_1})$ and $\mathcal
H=\mathcal G(\mathcal E_{\Ph_2})$ be the corresponding \'etale
groupoids. Then each homomorphism $\mathcal G\to \mathcal H$ induces
a morphism $f:X\to Y$ of the corresponding diffeological spaces
defined by $\mathcal E_{\Ph_1}$ and $\mathcal E_{\Ph_2}$.
Conversely, each morphism $f:X\to Y$ of the diffeological spaces
defines a homomorphism $\mathcal G\to\mathcal H$.
\end{theorem}
\begin{proof}
Let $P:\mathcal G\to\mathcal H$ be a homomorphism of smooth
groupoids. By definition of the actions of $\mathcal G$ and
$\mathcal H$ on $P$, for each $p\in P$, we have $\mathcal
Gs_P(p)=s_P(\mathcal Gp)$ and $t_P(p)\mathcal H=t_P(p\mathcal H)$.
Moreover, by (3) of \ref{Skandalis}, $s_P$ induces a surjective map
of orbit spaces $P/\mathcal G\to (\mathcal G_0/\mathcal G)$.

Let, for $p_1,p_2\in P$, we have  $s_P(\mathcal Gp_1)=s_P(\mathcal
Gp_2)$. Then there is a $g\in \mathcal G_1$ such that
$s_P(p_2)=s_P(gp_1)$ and by (3) of \ref{Skandalis} there is a unique
$h\in\mathcal H_1$ such that $gp_1=p_2h$. Therefore, by (1) of
\ref{Skandalis} we have $t_P(p_2\mathcal H)=t_P(gp_1\mathcal
H)=t_P(p_1\mathcal H)$. This implies that the correspondence
$s_P(\mathcal Gp)\to t_P(p\mathcal H)$ defines a map of orbit spaces
$f_P:\mathcal G_0/\mathcal G_1\to\mathcal H_0/\mathcal H_1$.

Put $X=\mathcal G_0/\mathcal G$ and $Y=\mathcal H_0/\mathcal H$ and
consider the natural structures of diffeological spaces on $X$ and
$Y$ defined by the groupoids $\mathcal G$ and $\mathcal H$.

Since, by (3) of \ref{Skandalis}, the $s_P$ is a surjective
submersion, for each $x\in\mathcal G_0$, there are local sections of
$s_P\to\mathcal G_0$ defined in sufficiently small neighborhoods of
$x$. Let $\si:U\to P$ be such a section. Then the composition
$t_P\o\si:U\to\mathcal H_0$ is a smooth map and its composition with
$\mathcal H_0\to\mathcal H_0/\mathcal H$ is a $\mathcal M$-chart on
$Y$ as a diffeological space. Let $\si_1,\si_2:U\to P$ be two such
sections. By (3) of \ref{Skandalis}, there is a smooth map
$h:U\to\mathcal H_1$ such that for each $y\in U$ we have
$\si_2(y)=\si_1(y)h(y)$. Then the $\mathcal M$-charts on $Y$ induced
by $\si_1$ and $\si_2$ as above coincide. Since the restrictions of
the projection $\mathcal G_0\to\mathcal G_0/\mathcal G$ to the open
subsets above form an $\mathcal M$-atlas on $X$ of the structure of
a diffeological space on $X$, the construction above induces a
morphism of diffeological spaces $X\to Y$.

Conversely, let  $\mathcal G$, $\mathcal H$ be two \'etale groupoids
determined by the locally invertible $\mathcal M_n$-atlases $\Ph_1$
on $X=\mathcal G_0/\mathcal G$ and $\Ph_2$ on $Y=\mathcal
H_0/\mathcal H$ and let $f:X\to Y$ be a morphism of the
corresponding diffeological spaces.

Put $\tilde P=\{(c,d,h):c\in \mathcal G_0,d\in \mathcal H_0,f\o
p_G(c)=p_H(d), h\in \mathcal H_1, t(h)=d\}$. Consider the following
equivalence relation on $\tilde P$: $(c,d,h)\sim (c,d',h'h)$, where
$h'\in\mathcal H_1$ and $h':d\to d'$, and the corresponding quotient
set $P=\tilde P/\sim$. Define a topology on $P$ as follows. For each
$c\in\mathcal G_0$, consider a restriction $p_U$ of a projection
$p_G:\mathcal G_0\to X$ to some neighborhood $U$ of $c$ as a
$\mathcal M$-chart on $X$. Then, if $U$ is sufficiently small, the
composition $f\o p_U$ is a $\mathcal D$-chart on $Y$ of the type
$p_H\o\vh_U$, where $\vh_U$ is a smooth map $U\to\mathcal H_0$. Let
$V$ be a neighborhood of $d=\vh_U(c)\in\mathcal H_o$ such that there
is a smooth section $\si:V\to\mathcal H_1$. We may assume that
$\vh_U(U)\subset V$. Consider a subset $P(U,\vh,\si)$ of $P$ defined
by the points of $\tilde P$ of the type $(c',\vh(c'),\si\o\vh(c'))$,
where $c'\in U$. Consider the topology on $P$ generated by the
subsets $P(U,\vh,\si)$ for all $c\in\mathcal G_0$, the maps $\vh$,
and the sections $\si$. Since each subset $P(U,\vh,\si)$ is a smooth
manifold, it is easy to check that there is a unique structure of a
smooth manifold on $P$ such that each $P(U,\vh,\si)$ is an open
submanifold of $P$.

Define the maps $\tilde s_P:\tilde P\to\mathcal G_0$ and $\tilde
t_P:\tilde P\to\mathcal H_0$ as follows: $\tilde s_P(c,d,h)=c$ and
$\tilde t_P(c,d,h)=s(h)$. It is easy to check that the maps $\tilde
s_P$ and $\tilde t_P$ induce the smooth maps $s_P:P\to\mathcal G_0$
and $t_P:P\to\mathcal H_0$. Consider the left action of $\mathcal G$
on $\tilde P$ and the right action of $\mathcal H$ on  $\tilde P$:
$g(c,d,h)=(t(g),d,h)$, where $g\in\mathcal G_1$ and $s(g)=c$, and
$(c,d,h)h'=(c,d,hh')$, where $h'\in\mathcal H_1$ and $t(h')=s(h)$.
It is easy to check that these actions induce the smooth actions on
$P$. It is clear that the map $s_P$ is $\mathcal H$-invariant, the
map $t_P$ is $\mathcal G$-invariant, the actions of $\mathcal G$ and
$\mathcal H$ commute, and $s_P:P\to\mathcal G_0$, as an $\mathcal
H$-bundle with the moment map $t_P$, is principal. Thus we have a
homomorphism of groupoids $\mathcal G\to\mathcal H$.
\end{proof}
\begin{remark*}
Let  $\vh:\mathcal G\to \mathcal H$ be a smooth covariant functor.
By definition, $\vh$ induces a map $f:\mathcal G_0/\mathcal
G\to\mathcal H_0/\mathcal H$ which is the morphism of the
corresponding diffeological spaces. It is easy to check that the
smooth morphism of groupoids $\mathcal G\to\mathcal H$, defined by
$P_\vh=\mathcal G_0\x_{\mathcal H_0}\mathcal
H_1=\{(x,h):\vh(x)=t(h)\}$, $s_P(c,h)=c$, $t_P(c,h)=s(h)$ and the
natural actions of $\mathcal G$ and $\mathcal H$ on $P_\vh$, is a
particular case of the construction of Theorem~\ref{homomorphism}.
\end{remark*}
The following corollaries are evident.

\begin{corollary}\label{Morita}
Let $\Ph_1$ and $\Ph_2$ be locally invertible $\mathcal M_n$-atlases
on the sets $X$ and $Y$ respectively and let $\mathcal G=\mathcal
G(\Ph_1)$ and $\mathcal H=\mathcal G(\Ph_2)$ be the corresponding
\'etale groupoids. Then the categories $\mathcal G$ and $\mathcal H$
are Morita equivalent iff the diffeological spaces $X$ and $Y$
defined by the $\mathcal M_n$-atlases $\Ph_1$ and $\Ph_2$ are
isomorphic as objects of the category $\mathcal M_{\on{sp}}$.
\end{corollary}

\begin{corollary}\label{Morita1}
Let $\Ph_1$ and $\Ph_2$ be equivalent locally invertible $\mathcal
M_n$-atlases on the set $X$. Then the corresponding \'etale
groupoids $\mathcal G(\Ph_1)$ and $\mathcal G(\Ph_2)$ are Morita
equivalent.
\end{corollary}

\section{Classifying topos}

Let $\Ph$ be an $\mathcal M_n$-atlas or $\mathcal M$-atlas on $X$
such that, for each $(M,k)\in\Ph$ and each open subset $U\subset M$
the restriction of $(M,k)$ with respect to the inclusion $U\subset
M$ belongs to $\Ph$. Let $S$ be a contravariant functor on the
category  $\mathcal C_{\Ph}$. For $(M,k)\in\Ph$, consider the
standard category $\mathcal C(M)$ of $M$ as topological space, i.e.
the category with objects open subsets $U$ of $M$ and morphisms the
inclusions $U_1\subset U_2$ of such open subsets. By assumption
$\mathcal C(M)$ is a subcategory of $\mathcal C_{\Ph}$. The functor
$S$ is called {\bfseries a sheaf on $\Ph$} if, for each
$(M,k)\in\Ph$, the restriction of $S$ to $\mathcal C(M)$ is a scheaf
on $M$. It is easy to see that if $\Ph$ is an invertible $\mathcal
M_n$-atlas on $X$, a sheaf on $\Ph$ defines a $\mathcal
C(\Ph)$-sheaf in sense of \cite{Mo} and, conversely, each $\mathcal
C(\Ph)$-sheaf induces a sheaf on $\Ph$. Therefore, the classifying
topos $\mathcal B(\mathcal C(\Ph))$ could be considered as the space
of sheaves on $\Ph$.

Let $\Ph$ be a $\mathcal M_n$-atlas on $X$ as above and let
$(X,\Psi)$ be the corresponding diffeological space. Since, by
definition, $\Psi=\bar\Ph$, using the inverse image functor one can
extend uniquely each scheaf on $\Ph$ to the scheaf on $\Psi$. This
remark implies the following theorem.
\begin{theorem}\label{scheaf}
Let $\Ph_1$ and $\Ph_2$ be locally invertible $\mathcal M_n$-atlases
on the sets $X$ and $Y$ respectively, $\mathcal G=\mathcal G(\Ph_1)$
and $\mathcal H=\mathcal G(\Ph_2)$ the corresponding \'etale
groupoids, $(X,\Psi_1)$ and $(Y,\Psi_2)$ the corresponding
diffeological spaces. Then each morphism $f:X\to Y$ of
diffeological spaces induces a morphism $\mathcal B(\mathcal
G)\to\mathcal B(\mathcal H)$ of toposes. If $f$ is an isomorphism,
it induces an isomorphism of the corresponding toposes.
\end{theorem}

\section{Nerve and classifying space for a smooth groupoid}

For a groupoid $\mathcal G$, denote by $\mathcal G_n$ the space of
composable strings of morphisms  in $\mathcal G_n$:
\begin{equation}\label{string}
x_0\stackrel{g_1}{\longleftarrow}x_1\stackrel{g_2}{\longleftarrow}...
\stackrel{g_n}{\longleftarrow}x_n.
\end{equation}

The spaces $\mathcal G_n$ ($n\geq 0$) with the face maps
$d_i:\mathcal G_n\to\mathcal G_{n-1}$ defined in the usual way:
$$
d_i(g_1,\dots,g_n)=
\begin{cases}
(g_2,\dots,g_n)&\text{if $i=0$}\\
(g_1,\dots,g_i g_{i+1},\dots,g_n)&\text{if $1\le i\le n-1$}\\
(g_1,\dots,g_{n-1})&\text{if $i=n$}
\end{cases}
$$
form a simplicial space, the {\bfseries nerve of} $\mathcal G$,
denoted by $N(\mathcal G)$. Its geometric realization is {\bfseries
the classifying space} of $\mathcal G$, denoted $B\mathcal G$.

For a smooth groupoid $\mathcal G$, the notation above agree with
the earlier notation for $n=0,1,2$. Thus, the space $\mathcal G_n$
is a fibered product $\mathcal G_1\x_{\mathcal G_0}\x...\x_{\mathcal
G_0}\mathcal G_1$ and, hence, has a structure of a smooth manifold.

 A Morita equivalence
$\vh:\mathcal G\stackrel{~}{\longrightarrow}\mathcal H$ induces a
weak homotopy equivalence $B\mathcal
G\stackrel{~}{\longrightarrow}B\mathcal H$ \cite{Hf1}.

Corollary \ref{Morita} and \cite{Mo} imply the following

\begin{theorem}\label{contraction}
Let $\Ph_1(X)$ and $\Ph_2(X)$ be two full $\mathcal M_n$-atlases on
$X$ such that $\Ph_1(X)\subset\Ph_2(X)$ and let $\mathcal
G({\Ph_1(X)})$ and $\mathcal G({\Ph_2(X)})$ be the correspponding
\'etale groupoids. Then the classifying spaces of $\mathcal
G(\Ph_1(X))$ and $\mathcal G(\Ph_2(X))$ are homotopy equivalent.
\end{theorem}

Let $\mathcal G$ be a \'etale groupoid such that its set of objects
$\mathcal G_0$ is an $n$-dimensional manifold $M$. For example, for
the Haefliger groupoid $\Ga_n$ we have $M=\mathbb R^n$. Let
$\mathcal F$ be a foliation on a manifold $N$ of codimension $n$ and
let $T$ be a complete transversal of $\mathcal F$. Consider the full
$\mathcal M_n$-atlas $\Ph$ on the set of leaves $N/\mathcal F$
determined by $T$ and the corresponding \'etale category $\mathcal
G(\Ph)$. For this category we have $M=T$.

In this case one can define the so-called embedding category
$\on{Emb}(\mathcal G)$. Objects of the category $\on{Emb}(\mathcal
G)$ are the elements of a fixed base $\mathcal U$ of the topology of
$M=G_0$ consisting of contractible open subsets. For two such open
subsets $U$ and $V$, each section $\si:U\to \mathcal G_1$ of the
source map $s$ with the property that $t\o\si:U\to \mathcal G_0$
defines an embedding $\hat\si:U\to V$, is a morphism of the category
$\on{Emb}(\mathcal G)$. The composition in $\on{Emb}(\mathcal G)$ is
defined by $\hat\ta\o\hat\si(x)=\ta(t\si(x))\cdot\si(x)$
(multipliation in $\mathcal G$). The nerve of the category
$\on{Emb}(\mathcal G)$ is the simplicial set, whose geometric
realization is the classifying space of $\on{Emb}(\mathcal G)$ and
is denoted by $B\on{Emb}(\mathcal G)$. There is the following

\begin{theorem} (\cite{Mo1})
For any \'etale groupoid $\mathcal G$, the classifying spaces
$B\mathcal{G}$ and $B\on{Emb}(\mathcal G)$ are weakly homotopy
equivalent.
\end{theorem}

In contrast to $B\mathcal G$, the classifying space
$B\on{Emb}(\mathcal G)$ is a $\on{CW}$-complex.

\begin{corollary}\label{weak}
The classifying spaces $B\mathcal G(\Ph)$ and $B\on{Emb}(\mathcal
G(\Ph))$ are weakly homotopy equivalent.
\end{corollary}

\section{The \v{C}ech-De Rham cohomology of $\mathcal M_n$-spaces}

 Let $X$ be an $\mathcal
M_n$-space defined by a full $\mathcal M_n$-atlas $\Ph(X)$ and let
$\mathcal G(\Ph(X))$ be the corresponding \'etale groupoid. The
space of objects of $\mathcal G(\Ph(X))$ is
 an $n$-dimensional manifold $M$. Let $\mathcal{U}$ be a fixed base
 of the topology on $M$. Consider the \v{C}ech complexes
for $\mathcal G(\Ph(X))$ as follows:
$$
\check C_\mathcal U(\Ph(X),\Om^p):
\prod\limits_{U_0}\Om^p(U_0)\stackrel{\de}{\longrightarrow}
\prod\limits_{U_0\stackrel{g_1}{\longrightarrow}U_1}\Om^p(U_0)
\stackrel{\de}{\longrightarrow}\prod\limits_{U_0\stackrel{g_1}{\longrightarrow}U_1
\stackrel{g_2}{\longrightarrow}U_2}\Om^p(U_0)\stackrel{\de}{\longrightarrow}...,
$$
where the product is taken over the strings of composable arrows of
$\mathcal U$, with the boundary
\begin{multline}
(\de\om)(g_1,\dots,g_{p+1})=g_1^*\om(g_2,\dots,g_{p+1})+\\
\sum_{i=1}^p(-1)^i\om(g_1,\dots,g_{i+1}g_i,\dots,g_{p+1})+(-1)^{p+1}\om(g_1,\dots
g_p).
\end{multline}

{\bfseries  The \v{C}ech-De Rham complex} $(\check C_\mathcal
U(\Ph(X),\Om^*),D)$ of a $\mathcal M_n$-space $X$ is the total
complex for the double complex $(\check C_\mathcal
U\Ph(X),\Om^*),\de,d)$, where $d$ is the De Rham differential and
$D=\de\pm d$ with the standard sign convention.
%We have a product
%$$
%(\om\cdot\eta)(\si_1,\dots,\si_{p+q})=(-1)^{p+q}\om(\si_1,\dots,\si_p)\cdot\si_1^* \dots
%\si_p^*\eta(\si_{p+1},\dots,\si_{p+q}),
%$$
%which makes $(\check C_{\mathcal U}(\Ph(X),\Om^*),D)$ into a bigraded differential algebra.

Denote by $\check H^*_{\mathcal U}(\Ph(X);\mathbb R)$ the cohomology
of the total complex $(\check C_\mathcal U(\Ph(X),\Om^*),D)$ and
call it the \v{C}ech-De Rham cohomology of the $\mathcal M_n$-space
$X$ (relative to $\Ph(X)$ and $\mathcal U$).

Consider the first spectral sequence of the double complex $(\check
C_\mathcal U(X,\Om^*),\de,d)$. It is clear that, for the spectral
sequence, we have $E_2^{pq}=E_\infty^{pq}=0$ for $q>0$ and
$E_2^{p0}=E_\infty^{p0}=H^p(B\mathcal G_\Ph;\mathbb R)$. Therefore,
we have an isomorphism $\check H^*_\Ph(X;\mathbb R)=H^*(B\mathcal
G_\Ph;\mathbb R)$. Thus, we have the following
\begin{theorem}
For an $\mathcal M_n$-space $X$ defined by a full $\mathcal
M_n$-atlas $\Ph(X)$, the \v{C}ech-De Rham cohomology $\check
H^*_{\mathcal U}(\Ph(X);\mathbb R)$ does not depend on the choice
neither of a full $\mathcal M_n$-atlas $\Ph(X)$, nor on a base
$\mathcal U$ and is denoted by $H^*(X;\mathbb R)$.
\end{theorem}

The cohomology $H^*(X;\mathbb R)$ is called the real cohomology of
$\mathcal M_n$-space $X$ and denoted by $H^*(X;\mathbb R)$.

The \v{C}ech-De Rham cohomology for a foliation are defined in
\cite{Cr}.

%%%%%%%%%%%%%%%%%%%%%%%%%%%%

\section {Categories $\mathcal P_n(G)$ and $\mathcal P_n(G)$-spaces}

Let $G$ be a Lie group.  Consider a category $\mathcal P_n(G)$ with
objects smooth principal $G$-bundles with $n$-dimensional bases and
morphisms the morphisms of such principal $G$-bundles which project
to the \'etale maps of bases. Applying the general procedure of
\cite{L6} to this category one obtains the category of $\mathcal
P_n(G)$-spaces. By definition, a $\mathcal P_n(G)$-chart on a set
$Y$ is a pair $(k,P)$, where $k:P\to Y$,  $P\in (\mathcal
P_n(G))_0$. Morphisms of such charts are defined by means of
morphisms of the category $\mathcal P_n(G)$. A $\mathcal
P_n(G)$-atlas on a set $Y$ is the set $\Psi$ of $\mathcal
P_n(G)$-charts on $Y$ such that the set of maps $k:J\o
I_{\Psi(Y)}(P,k)=J(P)\to Y$ ($(P,k)\in\Psi(Y))$) is an inductive
limit $\varinjlim J\o I_{\Psi(Y)}$ of the functor $J\o I_\Psi$,
where the functors $J$ and $I_{\Psi(Y)}$ are similar to those from
section \ref{secdefCsp}. Therefore, one defines a $\mathcal
P_n(G)$-space as a set $Y$ with  a maximal $\mathcal P_n(G)$-atlas
$\Psi_m(Y)$ on $Y$. For a principal smooth $G$-bundle $P$ with a
base $B$ the projection $p:P\to B$ is a covariant functor from
$\mathcal P_n(G)$ to $\mathcal M_n$. By the definition of the
category $\mathcal P_n$, for any $\mathcal P_n$-space $Y$ the
$\mathcal P_n$-atlas on $Y$ induces a $\mathcal M_n$-atlas on some
set $B_Y$ so that the extension of the functor $P\to B$ gives the
covariant functor $\tilde Y\to B_Y$ from the category of $\mathcal
P_n$-spaces to the category $\mathcal M_{n,\on{sp}}$. The $\mathcal
M_n$-set $B_Y$ is called the base of the $\mathcal P_n$-space $Y$.

A $\mathcal P_n(G)$-atlas $\Psi$ on a set $Y$ is called full if the
corresponding $\mathcal M_n$-atlas on the base $B_Y$ is full.

\begin{remark*}
1. For a $\mathcal M_n$-space $X$, consider the corresponding frame
bundle $Fr(X)$ (see section \ref{secstruc}). By definition, $Fr(X)$
is a $\mathcal P_n(\on{GL}(n,\mathbb R))$-space.

2. Consider a transversal principal $G$-bundle $P$ over a foliation
$(N,\mathcal F)$ of codimension $n$.  One can consider $P$ as a
principal $G$-bundle over $N$ with a smooth action of the holonomy
groupoid $\on{Hol}(\mathcal F)$ of the foliation $\mathcal F)$ on
$P$ which commutes with the action of $G$ on $P$ (see \cite{KT}). It
is evident that the action of $\on{Hol}(\mathcal F)$ on $P$ induces
an action on $N$. It is easy to construct a full $\mathcal
P_n(G)$-atlas on the set $Y$ of orbits of the groupoid
$\on{Hol}(\mathcal F)$ on $N$.
\end{remark*}

Let $G$ be a Lie group with Lie algebra $\mathfrak g$ and let
$p:P\to B$ be a smooth principal $G$-bundle. For $\xi\in\mathfrak
g$, denote by $X_\xi$ the corresponding vector field on $P$ induced
by the action of $G$ on $P$. Let $T_p$ be a tangent space at $p\in
P$ and let $l_p:T_p\to \mathfrak g$ be a linear map such that, for
any $\xi\in\mathfrak g$, we have $l_p(X_\xi(p))=\xi$. The space
$\tilde P$ of all such linear maps $l_p$ ($p\in P)$ is a smooth
fiber bundle over $P$ with the projection $l_p\to p$.

Consider a right action of $G$ on $\tilde P$ defined by: $g\to
\al(g)=\on{Ad}_{g^{-1}}\o l_p\o (R_g)_*^{-1}$, where $g\to R_g$ is
the action of $G$ on $P$.

The following lemma is evident.
\begin{lemma}\label{tilde P}
The fiber bundle $\tilde P\to P$ possesses the following properties:
\begin{enumerate}
\item the fibers of the fiber bundles $\tilde P\to P$ and $\tilde P/G\to B$ are affine spaces and, then,
the spaces $\tilde P/G$ and $B$ are homotopy equivalent;
\item the transformation $\al_g$ is an automorphism of the affine fiber bundle
$\tilde P\to P$;
\item the correspondence $g\to\al_g$ defines a free action of $G$ on $\tilde P$ and
$\tilde P\to \tilde P/G$ is a smooth principal $G$-bundle;
\item $\tilde P/G$ is a smooth fiber bundle over $B$ with affine fibers:
\item each connection form $\om$ on $P$ is a smooth $G$-invariant section of the fiber bundle $\tilde P\to P$
or a smooth section of the fiber bundle $\tilde P/G\to B$.
\end{enumerate}
\end{lemma}
Define a differential 1-form $\varkappa$ on $\tilde P$ with values
in $\mathfrak g$ as follows. For a tangent vector $w$ at
$l_p\in\tilde P$, put $\varkappa(w)=l_p(\tilde p_*w)\in\mathfrak g$.
The following lemma follows from the definitions directly.

\begin{lemma}\label{kappa}
The form $\varkappa$ have the following properties:
\begin{enumerate}
\item  $\varkappa$ is a connection form on the principal $G$-bundle $\tilde P$;
\item $\varkappa$ is invariant under the natural action of of the group of automorphisms of
the principal $G$-bundle $P$ on $\tilde P$;
\item Let $\om$ be a connection form on $P$ and $s:P\to\tilde P$ is the corresponding section.
Then we have $\om=s^*\varkappa$.
\end{enumerate}
\end{lemma}
The connection $\varkappa$ is called  {\bfseries a canonical
connection}. We denote by $K$ the curvature form of $\varkappa$.

Now we apply the Chern-Weil homomorphism to the canonical connection
$\varkappa$. Assume that the Lie group $G$ is reductive. It is known
that the space of invariant polynomials $I(\mathfrak g)$ on the Lie
algebra $\mathfrak g$ is a free algebra generated by a finite set of
homogeneous polynomials.

Let $F\in I(\mathfrak g)$ is a homogeneous polynomial of degree $p$.
Denote by $F(K)$ the differential $2p$-form on $\tilde P/G$ obtained
by the substitution of $K$ instead of a variable of $F$ with the
following alternation. The form $F(K)$ is a closed form on the space
$\tilde P/G$ and its cohomology class depends only on the structure
of the principal $G$-bundle $P$. This cohomology class is called
{\bfseries the characteristic class} of $P$ corresponding to the
polynomial $F$. Since, by (1) of lemma \ref{tilde P}, the spaces
$\tilde P/G$ and $B$ are homotopy equivalent, we obtain the
homomorphism $I(\mathfrak g)\to H^*(B,\mathbb R)$ which is called
the characteristic Chern-Weil homomorphism.

Let $\om$ be a connection form on the principal $G$-bundle $P\to B$
and $s:B\to \tilde P/G$ the corresponding section of the bundle
$\tilde P/G\to B$. Then $s^*F(K)=F(R)$, where $R$ is the curvature
form of $\om$, and the map $F\to F(K)$ defines the standard
characteristic Chern-Weil homomorphism associated with the
connection $\om$. It is clear that it does not depend on the choice
of $\om$.

Consider the covariant functors $P\to \tilde P$, $P\to B$, $P\to
\tilde P/G$, the functor morphisms $\tilde P\to\tilde P/G$ and
$\tilde P/G\to B$, and their extensions to the category of $\mathcal
P_n(G)$-spaces.

Moreover, consider the contravariant functors $P\to \Om^*(\tilde
P)$, $P\to\Om^*(\tilde P/G)$, and the functor morphism $\Om^*(\tilde
P/G)\to \Om^*(B)$. We have their extensions to the category of
$\mathcal P_n(G)$-spaces.

Thus, we get the de Rham complexes $\Om(\tilde Y)$, $\Om^*(\tilde
Y/G)$, and $\Om^*(B_Y)$ for any $\mathcal P_n(G)$-space $Y$ and
their cohomologies $H^*(\tilde Y)$, $H^*(\tilde Y/G)$, and
$H^*(B_Y)$. It is clear that the characteristic homomorphism
$I(\mathfrak g)\to \Om^*(\tilde P/G)$ is extended to the
characteristic homomorphisms $I(\mathfrak g)\to \Om^*(\tilde Y/G)$
and $I(\mathfrak g)\to H^*(\tilde Y/G)$ for any $\mathcal
P_n(G)$-space $Y$.

\begin{remark*}
The construction of the characteristic homomorphism $I(\mathfrak
g)\to \Om^*(\tilde P/G)$ is based on the statement (2) of lemma
\ref{kappa}. For a $\mathcal P_n(G)$-space $Y$, consider a morphism
$\tilde Y/G\to B_Y$ and the corresponding homomorphism $H^*(B_Y)\to
H^*(\tilde Y/G)$. In contrast to the isomorphism $H^*(B)\to
H^*(\tilde P/G)$ (see (1) of lemma \ref{tilde P}), the homomorphism
$H^*(Y_B)=H^*(\tilde Y/G)$ is not an isomorphism in general.
Therefore, the image of the characteristic homomorphism $I(\mathfrak
g)\to H^*(\tilde Y/G)$ cannot be projected to $H^*(Y_B)$.
\end{remark*}

Now we show how to extend the characteristic homomorphism to
$\mathcal P_n(G)$-spaces.

Let $\Psi$ be a full $\mathcal P_n(G)$-atlas on a $\mathcal
P_n(G)$-space $Y$ and let $\Ph(B_Y)$ be the corresponding $\mathcal
M_n$-atlas on the base $B_Y$. Let us take an $n$-dimensional
manifold $M$ and a base $\mathcal U$ of the topology of $M$
consisting of contractible open subsets of $M$. Assume that the
$\mathcal M_n$-atlas $\Ph(B_Y)$ is formed by the charts of the type
$(U,k)$, where $U\in\mathcal U$. For each $U\in\mathcal U$, let us
choose a $\mathcal P_n(G)$-chart $(P_U,\tilde k)\in\Psi$ which
induces a $\mathcal M_n$-chart $(U,k)$ on $B_Y$. Since $U$ is
contractible, for any $U,V\in\mathcal U$ and each embedding
$\si:U\to V$ there is a unique morphism of principal $G$-bundles
$P_U\to P_V$ which projects to $\si$.

Consider the \v{C}ech-De Rham complex $(\check C_\mathcal
U(\Ph(X),\Om^*),D)$ for $\Ph(B_Y)$. By the arguments above, each
string of composable morphisms of the $\mathcal M_n$-atlas
$\Ph(B_Y)$
$$
\check C_\mathcal U(\Ph(B_Y),\Om^p):
\prod\limits_{U_0}\Om^p(U_0)\stackrel{\de}{\longrightarrow}
\prod\limits_{U_0\stackrel{\si_1}{\longrightarrow}U_1}\Om^p(U_0)
\stackrel{\de}{\longrightarrow}\prod\limits_{U_0\stackrel{\si_1}{\longrightarrow}U_1
\stackrel{\si_1}{\longrightarrow}U_2}\Om^p(U_0)\stackrel{\de}{\longrightarrow}...,
$$
could be covered by the string of composable morphisms of $\mathcal
P_n(G)$-atlas on $Y$. Applying the extension of the functor $\tilde
P\to\tilde P/G$ to the latter string we get the string
$$
 \prod\limits_{U_0}\Om^p(P_{U_0})\stackrel{\de}{\longrightarrow}
\prod\limits_{U_0\stackrel{\si_1}{\longrightarrow}U_1}\Om^p(P_{U_0})
\stackrel{\de}{\longrightarrow}\prod\limits_{U_0\stackrel{\si_1}{\longrightarrow}U_1
\stackrel{\si_1}{\longrightarrow}U_2}\Om^p(P_{U_0})\stackrel{\de}{\longrightarrow}...
$$
and then we have the corresponding \v{C}ech-De Rham complex $(\check
C_\mathcal U(\Ph(X),\Om^*),D)$ for $\Ph(B_Y)$.

\section{The first Chern class for the Reeb foliation}

\subsection{The cohomologies $H^*(W_n)$ and $H^*(W_n,\text{GL}(n,\R))$}
Let $W_n$ be the algebra of formal vector fields in $n$ variables,
i.e. the topological vector space of $\infty$-jets at $0$ of smooth
vector fields on $\R^n$ with the bracket induced by the Lie bracket
of vector fields on $\R^n$. Consider $\R$ as a trivial $W_n$-module.
The complex $C^*(W_n)=\{C^q(W_n),d^q\}$ of standard continuous
cochains of $W_n$ with values in $\R$ is defined as follows:
$C^q(W_n)$ is the space of continuous skew-symmetric $q$-forms on
$W_n$ with values in $\R$ and the differential $d^q:C^q(W_n)\to
C^{q+1}(W_n)$ is defined by the following formula:
$$
(d^q c)(\xi_1,\dots,\xi_{q+1})=
\sum_{i,j}(-1)^{i+j}c([\xi_i,\xi_j],\xi_1,\dots,
\widehat{\xi_i},\dots,\widehat{\xi_j},\dots,\xi_{q+1}),
$$
where $c\in C^q(W_n)$, $\xi_1,\dots,\xi_{q+1}\in W_n$, and, as
usual, $\hat\xi$ means that the term $\xi$ is omitted.  We denote
the cohomology of this complex by $H^*(W_n)=\{H^p(W_n)\}$. The
natural action of $\text{GL}(n,\R)$ on $\R^n$ induces an action of
$\text{GL}(n,\R)$ on $C^*(W_n)$ by automorphisms of this complex.
Then we have the subcomplex of relatives cochains
$C^*(W_n,\text{GL}(n,\R))$ of $W_n$ with respect to
$\text{GL}(n,\R)$ consisting of $\text{GL}(n,\R)$-invariant cochains
from $C^*(W_n)$. We denote the cohomology of this complex by
$H^*(W_n,\text{GL}(n,\R))=\{H^p(W_n,\text{GL}(n,\R))\}$.

Recall some facts about the cohomology  $H^*(W_n,\on{GL}(n,\R))$
(\cite{BR}, \cite{F}, \cite{God}). By definition, $C^*(W_n)$ is a
graded differential algebra and the  differential $d$ is an
antiderivation of degree 1. For $\xi^i\in\R[[\R^n]]$ and
$\xi=\sum_{i=1}^n\xi^i\frac{\p}{\p x^i}\in W_n$, put
$$
c^i_{j_1\dots j_r}(\xi)=\frac{\p^r\xi^i}{\p x^{j_1}\dots\p
x^{j_r}}(0),
$$
where $x^i$ $(i=1,\dots,n)$ are the standard coordinates in $\R^n$.
By definition, we have  $c^i_{j_1\dots j_r}\in C^1(W_n)$. Moreover,
$c^i_{j_1\dots j_r}$  for $r=0,1,\dots$ and $i,j_1\dots
j_r=1,\dots,n$ are generators of the $DG$-algebra $C^*(W_n)$. Since
$d=\{d^q\}$ is an antiderivation of degree 1 of $C^*(W_n)$, it is
uniquely determined by the following conditions:
\begin{equation}\label{dc}
dc^i_{j_1\dots j_r}= \sum_{0\le k\le r}\sum_{s_1<\dots<
s_k}\sum_{l=1}^nc^i_{lj_1\dots\widehat{j_{s_1}}\dots
\widehat{j_{s_k}}\dots j_r}\wedge c^l_{j_{s_1}\dots j_{s_k}}.
\end{equation}
Put
$$
\ga=(c^i_j),\quad \Psi^i_j=\sum_{k=1}^nc^i_{jk}\wedge
c^k,\quad\on{and}\quad  \Psi=(\Psi^i_j).
$$
It is known that
$$
\Psi_p=\on{tr}(\underbrace{\Psi\wedge\dots\wedge\Psi}_{\text {p
times}})\quad (p=1,\dots,n)
$$
are cocycles of $C^*(W_n,\on{GL}_n(\R))$ and the cohomology classes
of these cocycles  generate $H^*(W_n,\on{GL}_n(\R))$. The cohomology
class of $\Psi_p$ is called $p$th formal Chern class.

\subsection{The space of frames of infinite order and the Gelfand-Kazhdan
form }\label{S(M)} Let $M$ be a connected orientable $n$-dimensional
smooth manifold. Denote by $S(M)$ the space of frames of infinite
order of $M$, i.e. $\infty$-jets at $0$ of germs at $0$ of smooth
regular at $0\in\R^n$ maps from $\R^n$ into $M$. It is known that
$S(M)$ is a manifold with model space $\R^\infty$ (\cite{BR}).

Define the canonical Gelfand-Kazhdan 1-form $\om$ with values in
$W_n$ on $S(M)$ (\cite{G-K} and \cite{BR}). Let $\tau$ be a tangent
vector at $s\in S(M)$ and  let $s(u)$ be a curve on $S(M)$ such that
$\tau=\frac{ds}{du}(0)$. One can represent $s(u)$ by a smooth family
$k_u$ of germs at $0$ of regular at $0\in\R^n$ maps $\R^n\to M$,
i.e. $s(u)=j^\infty_0k_u$. Then put
$$
\om(\tau)=-j_0^\infty\frac{d}{du}(k_0^{-1}\o k_u)(0).
$$

Let $c\in C^q(W_n)$. For each $s\in S(M)$ and $X_1,\dots,X_q\in
T_s$, put
$$
\om_c(X_1,\dots,X_q)=c(\om(X_1),\dots,\om(X_q)).
$$
It is known that $c\mapsto\om_c$ is a homomorphism of the complexes
$\al:C^*(W_n)\to\Om^*(S(M))$. Moreover, we have
$\al(C^*(W_n,\text{GL}(n,\R)))=\Om^*(S(M)/\text{GL}(n,\R))$.

It is easy to check that $\be=(\be^i_j)=-(\al(c^i_j))$ is a
connection form on a principal $\on{GL}_n(\R)$-bundle $S(M)\to
S(M)/\on{GL}_n(\R)$, $R=d\be+\be\wedge\be$ is the curvature form of
this connection and the image of the formal Chern class $\Psi_p$
under $\al$ equals $\al(\Psi_p)=Tr(\underbrace{R\wedge\dots\wedge
R}_{\mbox{$p$ times}})$.

{\bf Case $n=1$.} Now apply the constructions above to the case
$n=1$. Let $x$ be a coordinate on $M_1$ near $z_0\in M_1$ and
$k=k(u)$ be a smooth map $\R\to M_1$ regular at $0\in\R$. By
definition, $s=j_0^\infty k\in S(M_1)$. Thus, one can consider
$x_0=k(0),x_p=\frac{d^p k}{dt^p}(0)$ $(p=1,\dots)$ as coordinates on
$S(M_1)$, where $x_1\ne 0$. Let $s(u)$ be a curve in $S(M_1)$. It
can be obtained as follows: $s(u)=j^\infty_0k(u)(t)$, where for each
$u$, $k(u)(t)$ is a smooth map $\R\to M_1$ regular at $t=0$. Then
$\ta=\frac{d}{du}j_0^\infty k(u)(t)|_{u=0}$ is a tangent vector to
the curve $s(u)$ at $s(0)$. By definition, we have
$\om(\ta)=-j_0^\infty \frac{d}{du}( k^{-1}(0)\o k(u))|_{u=0}$.

Let $\om=(\om_0,\om_1,\om_2,\dots)$ and let $x_p$ $(p=0,\dots)$ be
the coordinates on $S(M_1)$ as above. By construction, we have
\begin{equation}
\begin{split}
\om_o&=-\frac{dx_0}{x_1},\quad \om_1=\frac{x_2dx_0}{x_1^2}-\frac{dx_1}{x_1},\\
\om_2=&\left(\frac{x_3}{2x_1^2}-\frac{x_2^2}{x_1^3}\right)dx_0+\frac{x_2}{x_1^2}dx_1-
\frac{dx_2}{2x_1}.
\end{split}
\end{equation}
 Consider the action of $\text{GL}(1,\R)=\R^*$ on $S(M_1)$. By definition, we have  for $\la\in\R^*$ and $s=(x_p)$: $\la s=(\la^px_p)$.
Then one can take $y_0=x_0,\,y_p=(x_1^p)^{-1}x_p$ ($p=2,\dots$) for
the coordinates on $S(M_1)/\text{GL}(1,\R)$. By definition, the
first Chern class $c_1$ in these coordinates is defined by the
2-form $c_{1,y}=dy_2\wedge dy_0$ on $S(M_1)/\text{GL}(1,\R)$. It is
clear that this form is exact if $M_1$ is a smooth manifold. But we
will show that, in general, this form is not exact if $M_1$ is a
${\mathcal D}_1$-space.

\subsection{Reeb's foliation}
Describe Reeb's foliation on the sphere $\mathbb S^3$. Consider the
sphere $\mathbb S^3$ in $\R^4=\R^2\x\R^2$ given by the equation
$|x|^2+|y|^2=2$, where $|x|$ is the standard norm in $\R^2$. Take
two subsets $\mathbb S^3_i$ $(i=1,2)$ of $\mathbb S^3$  defined by
the equations $|x|^2\le |y|^2$ and $|y|^2\le |x|^2$, respectively.
The projection $(x,y)\to x$ induces the structure of fiber bundle on
$\mathbb S^3_1$  with base $D_1=\{|x|^2\le 1\}$ and, for each $x\in
D_1$, the fiber is the circle $\mathbb S^1_x$ given by the equation
$|y|^2=2-|x|^2$. Therefore, $\mathbb S_1^3$ is diffeomorphic to
$D_1\x \mathbb S^1$. Similarly, $\mathbb S^3_2$ is diffeomorphic to
$D_2\x \mathbb S^1$, where $D_2=\{|y|^2\le 1\}$ and the fiber over
$y\in D_2$ is given by the equation $|x|^2=2-|y|^2$. By definition,
we have $\mathbb S^3=\mathbb S^3_1\cup \mathbb S^3_2$.

Consider a smooth function $f(t)$  on the interval $|t|<1$
satisfying the following conditions:
\begin{align}\label{f}
&f(0)=0,\,f(t)\ge 0,\,f(-t)=f(t),\\ \label{f1}\lim_{t\to\pm
1}&\frac{d^pf}{dt^p}(t)=\infty,\,\lim_{t\to\pm
1}\frac{d^p}{dt^p}\frac1{f'(t)}=0,\,\mbox{for $p=0,1,\dots$}
\end{align}

Consider the foliation of codimension $1$ on $\mathbb S_1^3$ with
the leaf $L=\mathbb S^1\x \mathbb S^1$ and the leaves $L_\al$ for
each $\al\in\R$, consisting of the points of the type $(x,|y|e^{2\pi
i(\al+f(|x|)})$, where $x$ is an interior point of $D_1$ and we
regard $\R^4$ as $\R^2\x\C$. It is clear that the leaf $L$ is
compact and the leaf $L_\al$ is diffeomorphic to $\R^2$.

Similarly, we define a foliation of codimension $1$ on $\mathbb
S^3_2$ using, in general, another function $f$. Combining these two
foliations we get a foliation on $\mathbb S^3$ called the Reeb
foliation and is denoted by $\mathfrak F$.

Consider the curves $\ga_1$: $\al\to (0,0,\sqrt 2e^{2\pi\al i})$
($\al\in\R)$ and $\ga_2$: $t\to\left(\frac{t}{\sqrt
2},\frac{t}{\sqrt 2},\sqrt{2-t^2}\right)$ $(|t|<\sqrt 2)$ on $S^3$.
By construction, the curves $\ga_1$ and $\ga_2$ are transverse to
Reeb's foliation on $\mathbb S^3$ and, then, can be considered as
$\mathcal D_1$-charts on the space of leaves $\mathbb S^3/\mathfrak
F$. Since $\ga_1(\al)\in L_\al$ and $\ga_2(t)\in L_\al$ for
$f(|t|)+\al=0$, the map $\vh$: $\al=-f(t)$ ($0<t<1)$ is a morphism
$\vh$ of ${\mathcal D}_1$-charts $\ga_2\to\ga_1$.

Denote by $\be_i$  and $y_i$ $(i\neq 1)$ the extension of the
coordinates $\al$ and $t$ to  the ${\mathcal D}_1$-space $S(\mathbb
S^3/\mathfrak F)/\text{GL}(1,\R)$. By definition, the $2$-form
$c_1$, defining the first Chern class, is given in these coordinates
by the forms $c_{1,\be}=d\be_2\wedge d\be_0$ and $c_{1,y}=dy_2\wedge
dy_0$, respectively. Evidently we have $\vh^*c_{1,\be}=c_{1,y}$.

Denote by $S^2(\mathbb S^3/\mathfrak F)$ the space of frames of the
second order of $\mathbb S^3/\mathfrak F$. Clearly, $c_1$ may be
considered as a 2-form on $S^2(\mathbb S^3/\mathfrak
F)/\text{GL}(1,\mathbb R)$.

\begin{theorem}\label{first}
The first Chern class $c_1$ is non-trivial in the
complex\footnote{In the original text it was stated that $c_1$ is
non-trivial in the complex $\Om^*(S(\mathbb S^3/\mathfrak
F)/\text{GL}(1,\mathbb R))$. Unfortunately, the original proof
contained a gap, on the other hand, it was easy to adapt it to the
proof of the current statement of the theorem. Note that if $M_1$ is
a manifold, then the cohomology of the complexes
$\Om^*(S(M_1)/\text{GL}(1,\mathbb R))$ and
$\Om^*(S^2(M_1)/\text{GL}(1,\mathbb R))$ are isomorphic. This is
generally not true if $M_1$ is a $\mathcal D_1$-space.}

$\Om^*(S^2(\mathbb S^3/\mathfrak F)/\text{\rm GL}(1,\mathbb R))$.
\end{theorem}

\begin{proof}

First we write the extension of the morphism $\vh$ to $S(\mathbb
S^3/\mathfrak F)$. Let $f(x)$ and $g(u)$ be two functions and let
$h=g\o f$. Recall the classical Fa\`a di Bruno formula for the $n$th
derivative $h^{(n)}$:
\begin{equation}\label{Faa}
h^{(n)}=n!\sum_{k=1}^n\frac{g^{(k)}\o
f}{k!}\sum_{i_1+\dots+i_k=n}\frac{f^{(i_1)}}{i_1!}\dots\frac{f^{(i_k)}}{i_k!}
\end{equation}

Applying \eqref{Faa} for $g=-f(t)$ and $f=t(u)$ we get
$$
\al_0=-f(t_0),\quad
\al_n=-n!\sum_{k=1}^n\frac{f^{(k)}}{k!}\sum_{i_1+\dots+
i_k=n}\frac{t_{i_1}}{i_1!}\dots\frac{t_{i_k}}{i_k!},\quad n\ge 1
$$
Then the extension of the morphism $\vh$ to $S(\mathbb S^3/\mathfrak
F)/\text{GL}(1,\R)$ is defined by the formulas
\begin{equation}\label{1}
\begin{split}
\be_0&=-f(y_0),\\
\be_n&=(-1)^{n-1}\left(n!\sum_{k=1}^{n-1}\frac{1}{k!}\frac{f^{(k)}(y_0)}{(f')^n(y_0)}\sum_{i_1+\dots
i_k=n}
\frac{y_{i_1}}{i_1!}\dots\frac{y_{i_k}}{i_k!}+\frac{f^{(n)}(y_0)}{(f')^n(y_0)}\right),
\end{split}
\end{equation}
where $y_1=1$, $n\ge 2$, $y_i$, and $\be_i$ are the coordinates on
$S(S^3/\mathfrak F)/\text{GL}(1,\R))$ corresponding to $t_i$ and
$\al_i$. In particular,  the extension of the morphism $\vh$ to
$S^2(\mathbb S^3/\mathfrak F)/\text{GL}(1,\R)$ is defined by the
formulas
$$\be_0=-f(y_0),\quad
\be_2=-\frac{y_2}{f'(y_0)}-\frac{f''(y_0)}{(f'(y_0))^2}.$$ By
induction with respect to $n$ from \eqref{Faa} for
$h(t)=\frac{1}{f'(t)}$ and \eqref{f}  we get for $n>1$
\begin{equation}\label{n}
\lim_{t\to\pm 1}\frac{f^{(n)}(t)}{(f'(t))^n}=0\quad\text{and}\quad
\lim_{t\to\pm 1}\frac{d}{dt}\frac{f^{(n)}(t)}{(f'(t))^n}=0.
\end{equation}

In particular, \eqref{n} implies $\lim_{y_0\to 1}\be_n=0$ for $n>1$.

Assume that the class $c_1$ is trivial. Then there is a form
$\ga\in\Om^1(S^2(\mathbb S^3/\mathfrak F)/\text{GL}(1,\mathbb R))$
such that $c_1=d\ga$. This means that, if $\ga_\be$ and $\ga_y$ are
the expressions of the form $\ga$ in the coordinates $\be_0,\be_2$
and $y_0,y_2$, respectively, we have $d\ga_\be=c_{1,\be}$,
$d\ga_{y}=c_{1,y}$, and $\vh^*\ga_\be=\ga_y$.

Let $$\ga_\be=\ga_0d\be_0+\ga_2d\be_2,$$ where $\ga_i$ are smooth
functions of $\be_0$, $\be_2$. So we have
$$\ga_\be=\be_2d\be_0+\la,$$ where $$\la=\la_0d\be_0+\la_2d\be_2$$ is
a closed $1$-form. Moreover, we have
$$
\vh^*\ga_\be=A_0dy_0+A_2dy_2,
$$
where $A_0$ and $A_2$ are smooth functions of $y_0$, $y_2$. By
definition, the $1$-form $\vh^*\ga$ has a smooth extension to a
neighborhood of each point $(1,y_2)$.

On the other hand, we have
\begin{equation}\label{2}
\vh^*\ga_\be=\left( y_2+\frac{f''}{f'}-(\la_0\o\vh)\cdot
f'\right)dy_0+(\la_2\o\vh)\cdot d(\be_2\o\vh),
\end{equation}
where $f=f(y_0)$. From \eqref{2} we get
\begin{multline}\label{3}
A_0= y_2+\frac{f''}{f'}-(\la_0\o\vh)\cdot f'-(\la_2\o\vh)\cdot\left(
\left(\frac{f''}{(f')^2}\right)'+y_2\left(\frac{1}{f'}\right)'
 \right).
\end{multline}
Fix a $\be_0$. There exists a sequence $\{y_{0,n}\}$ such that
$\lim_{n\to\infty}y_{0,n}= 1$ and $-f(y_{0,n})\equiv\be_0\mod 1$. It
is clear that
$$\lim_{n\to\infty}\frac1{f'(y_{0,n})}A_0(y_{0,n},y_2)=0.$$ By this,
\eqref{3}, \eqref{f}, and \eqref{n}, we get $\la_0(\be_0,0)=0$.
Therefore, we have
\begin{equation}\label{4}
\la_0(\be_0,\be_2)=\mu(\be_0,\be_2)\be_2,
\end{equation}
where $\mu(\be_0,\be_2)$ is a smooth function of $\be_0,\be_2$ and
$\mu(\be_0,0)=\frac{\p\la_0}{\p\be_2}(\be_0,0)$.

Consider $\lim_{n\to\infty}\frac{\p A_{0}}{\p y_2}( y_{0,n},y_2)$.
Since  $\lim_{y_0\to 1}f'\frac{\p\be_2}{\p y_2}=-1$, we get
$$
\frac{\p A_0}{\p y_2}( 1,y_2)=1+\frac{\p\la_0}{\p \be_2}(\be_0,0).
$$
Since the form $\la$ is closed, we have
\begin{equation}\label{5}
\frac{\p\la_2}{\p\be_0}(\be_0,0)=\frac{\p\la_0}{\p\be_2}(\be_0,0)=c=\text{const}.
\end{equation}
By definition, the form $\la$ is periodic with respect to $\be_0$
with period $1$. This implies that $\la_2$ and
$\frac{\p\la_2}{\p\be_0}$ are periodic with respect to $\be_0$ with
the same period as well. It is easy to see that, if the derivative
of a smooth periodic function is constant, this constant equals $0$.
Then \eqref{5} implies $\frac{\p\la_0}{\p\be_2}(\be_0,0)=c=0$. Thus,
$\mu(\be_0,\be_2)=\nu(\be_0,\be_2)\be_2$, where $\nu(\be_0,\be_2)$
is a smooth function of $\be_0$ and $\be_2$, and, by \eqref{4}, we
have $\la_0(\be_0,\be_2)=\nu(\be_0,\be_2)\be_2^2$.

Put now in  \eqref{3} $y_2=0$. We get
\begin{multline*}
A_0(y_0,0)=
\frac{f''}{f'}-\nu\left(-f,-\frac{f''}{(f')^2}\right)\frac{(f'')^2}{(f')^3}-\la_2\left(-f,-\frac{f''}{(f')^2}\right)\left(\frac{f''}{(f')^2}\right)',
\end{multline*}
 where $f=f(y_0)$. This equation can rewritten in the following
form
$$
\frac{f''}{f'}=\frac{A_0(y_0,0)+\la_2\left(-f,-\frac{f''}{(f')^2}\right)\left(\frac{f''}{(f')^2}\right)'}{1-\nu\left(-f,-\frac{f''}{(f')^2}\right)\frac{f''}{(f')^2}}.
$$

Consider a sequence $\{y_{0,n}\}$ as above and substitute it to the
last equality. The right hand side of the equality tends to $A(1,0)$
as $n$ tends to $+\infty$. The properties \eqref{f1} of the function
$f$ imply that $\frac{f''}{f'}$ tends to $+\infty$ as $n$ tends to
$+\infty$. We get a contradiction. This contradiction proves the
theorem.
\end{proof}

Now we study the geometrical meaning of the first Chern class of
Reeb's foliation\footnote{On April 19 of 2013 M.V.\,Losik informed
A.\,Galaev in an email that the first Chern class defines the Reeb
foliation up to the orientation. Unfortunately, the proof of this
statement is missing. The following sentence suggests that the idea
of the proof was to show that the first Chern class determines the
holonomy group of the compact leaf.}. We recall that Reeb's
foliation is determined up to isomorphism by the holonomy group of
the compact leaf $\mathbb S^1\x \mathbb S^1$ (see \cite{Ser}) and by
some considerations on orientation studied in details in \cite{Miz}.

%\newpage

\end{document}